\documentclass{amsart}
\usepackage{graphicx}
\usepackage{tikz}
\usepackage{amsfonts}
\usepackage{amssymb}
\usepackage{amstext}
\usepackage[T1]{fontenc}
\usepackage[]{amscd}

\usepackage{hyperref}

\newtheorem{theo}{Theorem}
\newtheorem{prop}{Proposition}
\newtheorem{lem}{Lemma}
\newtheorem{cor}{Corollary}

\newtheorem{rem}{Remark}
\newtheorem{acknowledgement}{Acknowledgement}
\begin{document}

\title{Automorphisms of certain Niemeier lattices\\ and \\Elliptic Fibrations 
 }

\author{Marie Jos\'e Bertin and Odile Lecacheux}


\email{marie-jose.bertin@imj-prg.fr\\
odile.lecacheux@imj-prg.fr}
\date{\today}

\curraddr{Universit\'e Pierre et Marie Curie (Paris 6), Institut
de Math\'ematiques, 4 Place Jussieu, 75005 PARIS, France}

\begin{abstract}
Nishiyama introduced a lattice theoretic classification of the elliptic fibrations on a $K3$ surface. In a previous paper we used his method to exhibit $52$ elliptic fibrations, up to isomorphisms, of the singular $K3$ surface of discriminant $-12$. We prove here that the list is complete with a $53$th fibration, thanks to a remark of Elkies and Sch\"{u}tt. We characterize the fibration both theoretically and with a Weierstrass model.

\end{abstract}

\maketitle
\section{Introduction}

In a previous paper \cite{BGL}, the authors gave a classification, up to automorphisms, of the elliptic fibrations on the singular $K3$ surface $X$ whose transcendental lattice is isometric to $\langle 6 \rangle \oplus \langle 2 \rangle$. This classification was derived from the Kneser-Nishiyama method. Each elliptic fibration was given with the Dynkin diagrams characterizing its reducible fibers, the rank and torsion of its Mordell-Weil group. Hence $52$ elliptic fibrations were obtained.

Later on, Elkies and Sch\"{u}tt informed us that we missed an elliptic fibration.
More precisely, Elkies said  how he discovered the lack \cite{El}: ``while tabulating some information about the lattices in this genus (positive-definite even lattice of rank $18$ and discriminant $12$)...I had already done the smaller discriminants), including the sizes of their automorphism groups, and calculated their total mass (=sum of $1/|\text{Aut}(G)|$) which added up to less than the prediction of the mass formula. The discrepancy was a fraction $1/N$ so I guessed that just one lattice, with $N$ automorphisms, was missing, and eventually figured out where I lost the $53$rd lattice.''
 This paper intends to complete the gap.

Let us recall briefly the context. Given $\mathcal E$ an elliptic fibration on $X$, we define its trivial lattice by ${\text{T}}(\mathcal{E}):=U\oplus (W_{\mathcal{E}})_{\text{root}}$ where $W_{\mathcal{E}}$ denotes its frame lattice, that is the orthogonal complement of $U$ in the Neron-Severi group $NS(X)$. The Mordell-Weil group of $\mathcal E$ is encoded in the frame
\begin{equation}\label{F:1}
MW(\mathcal E)=W_{\mathcal E}/(W_{\mathcal E})_{\text{root}}.
\end{equation}

Thus 
\begin{equation}\label{F:2}
\text{rk}(MW(\mathcal E))=\text{rk}(W_{\mathcal E})-\text{rk}(W_{\mathcal E})_{\text{root}} \,\,\,\,\,\,\,\,(MW(\mathcal E))_{\text{tors}}=\overline{(W_{\mathcal E})_{\text{root}}}/(W_{\mathcal E})_{\text{root}}.
\end{equation}

The Kneser-Nishiyama's method provides a determination of the frame. Starting from the transcendental lattice of $X$
\[
T_X=
\begin{pmatrix}
  6  &  0 \\
  0  &  2
\end{pmatrix},
\]
denote $T$ the root lattice $T=A_5\oplus A_1$, orthogonal complement of $T_X(-1)$ in the root lattice $E_8$.
Take a Niemeier lattice $L$, that is a unimodular lattice of rank $24$, with root lattice $L_{\text{root}}$, often written $L=N(L_{\text{root}})$. Consider a primitive embedding $\phi: T \hookrightarrow L$. The orthogonal complement of $\phi(T)$ in $L$ is the frame of an elliptic fibration on $X$ and since $T$ is a root lattice \cite{BGL}, it suffices to consider all the primitive embeddings of $T$ in $L_{\text{root}}$ to obtain all the elliptic fibrations on $X$. Denote

\[ W=(\phi(A_5 \oplus A_1))^{\perp _{L}} \qquad \text{and} \qquad  N=(\phi(A_5 \oplus A_1))^{\perp _{L_{\text{root}}}}\]

and observe that $W_{\text{root}}=N_{\text{root}}$. Moreover the trivial lattice of the elliptic fibration provided by $\phi$ satisfies $T(\mathcal E)=U \oplus W_{\text{root}}$ and we can apply formulae (\ref{F:1}) and (\ref{F:2}).

Now given two points $P$ and $Q$ of the Mordell-Weil group, we can define a height pairing. The Mordell-Weil group, up to its torsion subgroup, equipped with this height pairing, is the Mordell-Weil lattice $MWL(X)$ which satisfies
\[ MWL(X)=W/\overline{W_{\text{root}}}.\]
Thus we recover more than the rank and torsion but also torsion and infinite sections of the elliptic fibration.

To list all the primitive embeddings of $A_5 \oplus A_1$ in the various Niemeier lattices, the authors of \cite{BGL} used Nishiyama's tables \cite{Nis} p.309 and p.323. They noticed two primitive embeddings of $A_5$ in $D_6$, not isomorphic by the Weyl group of $D_6$, namely
\[i_1(A_5)=(d_5,d_4,d_3,d_2,d_1) \qquad \text{and} \qquad i_2(A_5)=(d_6,d_4,d_3,d_2,d_1)\]
but p.323, Nishiyama missed the orthogonal complement in $D_6$ of $i_1(A_5)$. That is the origin of the gap which concerns the primitive embeddings of $A_5 \oplus A_1$ in $L=N(D_6^4)$ and $L=N(A_9^2D_6)$.

The paper is divided in two parts. In the first part we prove that the two primitive embeddings of $A_5$ in $D_6$ give two primitive embeddings of $A_5\oplus A_1$ in $N(D_6^4)$ isomorphic by an element of $\text{Aut} (N(D_6^4))$ so lead to just one elliptic fibration up to isomorphism. On the contrary, these embeddings $i_1$ and $i_2$ give rise to two non isomorphic primitive embeddings in $N(A_9^2D_6)$ thus exactly to two elliptic fibrations and not only one as listed in \cite{BGL}. Hence we obtain the $53$th fibration denoted by $\#40$ bis. We also explain the determination of the Mordell-Weil lattices.

In the second part we show how to derive the corresponding elliptic fibrations from the fibration $\#50$ of \cite{BGL} with Weierstrass equation (\ref{Eq:Eu}) and its associated graph. We set also the correspondence betweeen the results found in the first part of the paper and those coming from the graph.

\begin{acknowledgement} We are grateful to N. Elkies and M. Sch\"{u}tt for pointing out a missing fibration in the classification \cite{BGL}.
\end{acknowledgement}

\section{Some facts concerning Niemeier lattices and their automorphisms}

Concerning the definitions and properties of the irreducible root lattices $A_n$, $D_n$, $E_n$ and their dual lattices we refer to \cite{BGL} or \cite{BL} and use Bourbaki's notations, as in the Dynkin diagram of $D_6$ (see section 3).

Let $L$ a Niemeier lattice i.e. a unimodular lattice of rank $24$. We define its root lattice $L_{\text{root}}=\{\alpha \in L / <\alpha,\alpha>=-2\}$ where $<.,.>$ denotes the $\mathbb Z$-bilinear form on $L$. We recall that a Niemeier lattice $L$ is, up to an isomorphism, entirely determined by its root lattice $L_{\text{root}}$; thus it is denoted $L=N(L_{\text{root}})$. It can be realized as a sublattice of the dual lattice $(L_{\text{root}})^*$ of $L_{\text{root}}$. Thus $N(L_{\text{root}})/L_{\text{root}}$ is a finite abelian group, called the ``glue code'' or the set of ``glue vectors''. Writing $L_{\text{root}}=L_1\oplus L_2 ...\oplus L_k$ where the $L_i$ are irreducible root lattices of type $A_n$, $D_n$ or $E_n$, a typical glue vector of $L$ can be written \cite{CS}, 
\begin{equation}\label{E:1}
z=[y_1,y_2, ... ,y_k]
\end{equation}

where $y_i$ is a member of the dual lattice $L_i^*$. Any $y_i$ can be altered by adding a vector of $L_i$ so we may suppose that $y_i$ belongs to a standard system of representatives for the cosets of $L_i$ in $L_i^*$. It is usual to choose the glue vectors to be of minimal length in their cosets.

The various vectors $z$ of (\ref{E:1}) must have integral inner products with each other and be closed under addition modulo $L_1\oplus ... \oplus L_k$. This process is called ``gluing'' the components $L_1$, ... $L_k$.  

\subsection{The automorphism group $\text{Aut}(L_{\text{root}})$}
In the sequel we denote $X \rtimes Y$ a split extension of a group $Y$ by a group $X$.
We recall that
\[\text{Aut}(L_m)=W(L_m)\rtimes G_1(L_m)\]
where $W(L_m)$ is the Weyl group of $L_m$ and $G_1(L_m)$ the subgroup of $\text{Aut}(L_m)$ consisting of all Dynkin diagram automorphisms of $L_m$.

Set $G_0(L_{\text{root}}):=\prod_{m=1}^{k}W(L_m),\,\,\,\,\,\,\,\,G_1(L_{\text{root}}):=\prod_{m=1}^{k}G_1(L_m)$ and $K(L_{\text{root}})$ the following subgroup of  $\text{Aut}(L_{\text{root}})$
\[K(L_{\text{root}}):=\{\tau \in  \text{Aut}(L_{\text{root}}) / \tau(L_m)=L_m \,\,\,\forall\,\, m,\,\,\, 1\leq m \leq k \}.\]
The group  $G_0(L_{\text{root}})$ is called the Weyl group of $L_{\text{root}}$ and is a normal subgroup of $K(L_{\text{root}})$. The group  $G_1(L_{\text{root}})$ is a subgroup of $K(L_{\text{root}})$ and we have the relation
\[K(L_{\text{root}})=\prod_{m=1}^k \text{Aut}(L_m)=G_0(L_{\text{root}}) \rtimes G_1(L_{\text{root}}).\]
For each $1\leq i<j\leq k$ such that $L_i \simeq L_j$, denote $t_{ij}$ the transposition between the entries $i$ and $j$ and set 
\[G_2(L_{\text{root}}):=\langle t_{ij} / 1\leq i<j\leq k \,\,\,\,\,L_i\simeq L_j \rangle\]
the subgroup of $\text{Aut}(L_{\text{root}})$ of all permutations of the concerned entries. Finally we get 
\[\text{Aut}(L_{\text{root}})=K(L_{\text{root}})\rtimes G_2(L_{\text{root}})=(G_0(L_{\text{root}})\rtimes G_1(L_{\text{root}}))\rtimes G_2(L_{\text{root}}).\]

\subsection{The automorphism group $\text{Aut}(L)$}
Since the spanning set $\Delta=\{\alpha \in L /\,\,\,\langle\alpha, \alpha \rangle =-2 \}$ of $L_{\text{root}}$ is stable under the action of $\text{Aut}(L)$, it follows that $L_{\text{root}}$ is stable under $\text{Aut}(L)$ and we get a group homomorphism
\[
\begin{matrix}
\text{Aut}(L) & \rightarrow & \text{Aut}(L_{\text{root}})\\
     \tau & \mapsto & \tau |_{L_{\text{root}}}
\end{matrix}.
\]

Set $G_0(L):=G_0(L_{\text{root}})$ ; it is a normal subgroup of $\text{Aut}(L)$. Define the subgroup of $\text{Aut}(L)$,  $G_1(L):=\text{Aut}(L) \cap G_1(L_{\text{root}})$. They satisfy the relation
\[K(L_{\text{root}})\cap \text{Aut}(L)=G_0(L) \rtimes G_1(L).\]
Defining the subgroup $H(L)$ of $\text{Aut}(L)$ by $H(L):=\text{Aut}(L) \cap (G_1(L_{\text{root}}) \rtimes G_2(L_{\text{root}}))$, it follows $\text{Aut}(L) =G_0(L) \rtimes H(L)$.
Define the subgroup $G_2(L)$ of $G_2(L_{\text{root}})$ by
\[G_2(L):=\{\tau \in G_2(L_{\text{root}}) / \tau_1 \tau \in H(L) \,\,\,\, \text{for some   } \tau_1 \in G_1(L_{\text{root}})\}.\]
From this definition we get a surjective homomorphism $\pi_2$
\[
\begin{matrix}
\pi_2: & H(L) &  \rightarrow & G_2(L) \\
       & \tau & \mapsto & \tau_2
\end{matrix}
\]
and the exact sequence
\begin{equation}\label{Eq:suite}
 1 \rightarrow G_1(L) \rightarrow H(L) \rightarrow G_2(L) \rightarrow 1.
\end{equation}

Because $\text{Aut}(L)$ is a subgroup of $\text{Aut}(L_{\text{root}})$, we get the induced action of $\text{Aut}(L)$ on the ``glue code'' $L/L_{\text{root}}$. Moreover this action is the identity if and only if the element $\tau$ of $\text{Aut}(L)$ belongs to $G_0(L)$.
Finally we observe that $H(L)$ is identical to the subgroup of $G_0(L_{\text{root}})\rtimes G_1(L_{\text{root}})$ consisting of the elements preserving the ``glue code''.

For more details explaining how $\text{Aut}(L)$ is obtained from $\text{Aut}(L_{\text{root}})$ and how we can construct an automorphism of $L$, we refer to \cite{IS1} and \cite{IS2}.

\section{The Niemeier lattice $N(D_6^4)$}
Recall first the glue vectors of $D_6$. They are denoted $[0]$, $[1]$, $[2]$, $[3]$ by Conway and Sloane \cite{CS} and $\delta_6$, $\bar{\delta_6}$, $\tilde{\delta_6}$ in \cite{BGL} with the following correspondance

\begin{minipage}[t]{0.645\textwidth}
\begin{center}
$\begin{matrix}
[1] & =\delta_6  &=\frac{1}{2} (d_1+2d_2+3d_3+4d_4+2d_5+3d_6)\\
[2] & =\bar{\delta_6}  &= d_1+d_2+d_3+d_4+\frac{1}{2}(d_5+d_6)\\
[3] & =\tilde{\delta_6}  &=\frac{1}{2} (d_1+2d_2+3d_3+4d_4+3d_5+2d_6),
\end{matrix}$
\end{center}
and satisfy $[1]+[3]=[2]$.
\end{minipage}
\begin{minipage}[t]{0.345\textwidth}
\begin{center}
\begin{tikzpicture}[scale=0.3]%
\draw[fill=black](0.1,1.7039)circle (0.0865cm)node [below] {$d_6$};
\draw[fill=black](2.9,1.7039)circle (0.0865cm)node [below] {$d_{4}$};
\draw[fill=black](5.2,1.7039)circle (0.0865cm)node [below] {$d_{3}$};
\draw[fill=black](9.9,1.7039)circle (0.0865cm)node [below]{$d_1$};
\draw[fill=black](7.6,1.7039)circle (0.0865cm)node [below]{$d_2$};
\draw[fill=black](2.9,3.10)circle (0.0865cm)node [above]{$d_{5}$};
\draw[thick](0.1,1.7039)--(9.9,1.7039);
\draw[thick](2.9,1.7039)--(2.9,3.10);
\end{tikzpicture}
\end{center}
\end{minipage}

Also $\text{Aut}(D_6)=W(D_6)\rtimes G_1(D_6)$ with $G_1(D_6)\simeq \mathbb Z/2\mathbb Z$ which interchanges the glue vectors $[1]$ and $[3]$.

Moreover
\[N(D_6^4)=\mathbb Z\{D_6\oplus D_6 \oplus D_6 \oplus D_6, \text{glue code} \}.\]
The glue code, i.e. the set of glue vectors is given by all the even permutations of $[0,1,2,3]$ where $i$ denotes, by abuse of notation, the glue vector $[i]$. Thus $\mathcal A_4$ is contained in $\text{Aut}(N(D_6^4))$. More explicitly the glue code is

\begin{equation}\label{glue:G}
\begin{matrix}
[0,0,0,0], & [0,1,2,3], & [0,3,1,2], & [0,2,3,1], \\
[1,1,1,1], & [1,0,3,2], & [1,3,2,0], & [1,2,0,3], \\
[2,2,2,2], & [2,0,1,3], & [2,3,0,1], & [2,1,3,0],\\
[3,3,3,3], & [3,0,2,1], & [3,1,0,2], & [3,2,1,0].
\end{matrix}
\end{equation}

\begin{lem}
Up to an isomorphism of the Weyl group $W(D_6)$, there are two primitive embeddings of $A_5$ in $D_6$, namely
\[i_1(A_5)=(d_5,d_4,d_3,d_2,d_1)\]
\[i_2(A_5)=(d_6,d_4,d_3,d_2,d_1).\]
These two embeddings are interchanged by the element $g\in G_1(D_6)$ interchanging $d_5$ and $d_6$. Moreover $g$ acts on the glue vectors of $D_6$:
\[ g([1])=[3],\,\,\,\,\,g([2])=[2],\,\,\,\,\,g([3])=[1].\]
\end{lem}
\begin{proof}
It follows straightforward from the definitions.
\end{proof}

\begin{theo}
Let $y$ be any glue vector of $N(D_6^4)$, $y=[a,b,c,d]$. Define the application of $g$ on the glue code as $g(y)=[g(a),g(b),g(c),g(d)]$. Denote by $\tau$ any transposition of two components. Then $\tau \circ g \in \text{Aut}(N(D_6^4))$.
\end{theo}

\begin{proof}
Consider any permutation of two elements, for example take for $\tau$ the transposition of the two last components.
Observe first that $\tau$ and $g$ commute; it follows $(\tau \circ g)^2 =\text{Id}$. This allows us to present the action of $\tau \circ g$ on the glue code as below:
\[
\begin{matrix}
[0,0,0,0] & [0,1,2,3] & [0,2,3,1] & [1,1,1,1] & [1,3,2,0] \\
\tau \circ g \updownarrow & \updownarrow & \updownarrow & \updownarrow & \updownarrow \\
[0,0,0,0] & [0,3,1,2] & [0,2,3,1] & [3,3,3,3] & [3,1,0,2] \\

\bigskip\\

[2,2,2,2] & [2,0,1,3] & [2,1,3,0] & [1,0,3,2] & [1,2,0,3] \\
\tau \circ g \updownarrow & \updownarrow & \updownarrow & \updownarrow & \updownarrow \\
[2,2,2,2] & [2,0,1,3] & [2,3,0,1] & [3,0,2,1] & [3,2,1,0].
\end{matrix}
\]





Since $\tau \circ g$ is bijective on the glue code it belongs to $\text{Aut}(N(D_6^4))$. The same conclusion is obtained if $\tau$ is an arbitrary transposition.

\begin{rem}
The well-known isomorphism $G_1(N(D_6^4))\rtimes G_2(N(D_6^4)) \simeq \mathcal S_4$ \cite{CS} can be explicited as 
\[
\begin{matrix}
\mathcal S_4 & \rightarrow & \text{Aut}(N(D_6^4)) \\
\sigma & \mapsto & \sigma \circ g^{e(\sigma)}
\end{matrix}
\]
where $e(\sigma)=0$ if $\sigma$ is even and $1$ otherwise.

\end{rem}

\begin{rem}
 Moreover if $\tau$ permutes the two last components, $\tau \circ g$ fixes the glue vectors having their two first components made with $0$ or $2$, permutes the glue vectors beginning by $0$ on one side and the glue vectors beginning by $2$ on the other side; also it transforms the glue vectors beginning by $1$ into the glue vectors beginning by $3$. 
\end{rem}

\end{proof}

\begin{cor}
The two primitive embeddings of $A_5 \oplus A_1$ in $N(D_6^4)$ given by

$(i_1(A_5),d_6,0,0)$ and $(i_2(A_5),d_5,0,0)$ are isomorphic by an element of $\text{Aut}(N(D_6^4))$.
\end{cor}
\begin{proof}
We take for $\tau$ the transposition of the two last components. We get that $\tau \circ g$ interchanges the two embeddings and by the previous theorem belongs to $\text{Aut}(N(D_6^4))$.
\end{proof}

\section{The Niemeier lattice $N(A_9^2D_6)$}

Aside the glue vectors of $D_6$ defined in the previous section, the glue group of $A_9$ is cyclic, generated by $\alpha$, see for example \cite{BL} or \cite{CS}:
\[\alpha=\frac{1}{10}[9a_1+8a_2+7a_3+6a_4+5a_5+4a_6+3a_7+2a_8+a_9].\]
By abuse of notation we write $1$ for the class of $\alpha$ in $A_9^*/A_9$ and more generally $i$ for the class of $i\alpha$. We recall that 
\[ \text{Aut}(A_9)=W(A_9)\rtimes G_1(A_9),\]
where $W(A_9)$ denotes the Weyl group and $G_1(A_9)$ consists in the automorphisms of the Dynkin diagram of $A_9$ forming a group of order $2$ exchanging $a_i$ and $a_{10-i}$ for all $1\leq i \leq 9$ and therefore $i$ and $10-i$ according to the above convention. This automorphism acting on the first (resp. second) factor  $A_9$ of $L_{\text{root}}$ will be denoted $\gamma_1$ (resp. $\gamma_2$). It follows
\[G_1(L_{\text{root}})=G_1(A_9^{(1)}A_9^{(2)}D_6)\simeq \mathbb Z/2\mathbb Z\times \mathbb Z/2\mathbb Z \times \mathbb Z/2\mathbb Z,\]
\[G_2(L_{\text{root}})\simeq \mathbb Z/2\mathbb Z =\langle h \rangle,\]
where $h$ exchanges the two copies of $A_9$.

Set $\gamma=\gamma_1 \gamma_2$, $h_1=\gamma_1 g$ and $h_2=\gamma_2 g$.

\begin{prop}
\begin{enumerate}
\item The subgroup $G_1(L)=\text{Aut}(L)\cap G_1(L_{\text{root}})\simeq \mathbb Z/2\mathbb Z$ is generated by $\gamma$.
\item The automorphism $h$ of $G_2(A_9^{(1)}A_9^{(2)}D_6)$ is an automorphism of $G_2(L)=G_2(N(A_9^{(1)}A_9^{(2)}D_6))$; moreover $h_1h$ and $h_2h$ belong to $\text{Aut}(L)$. Hence the subgroup $G_2(L)$ is generated by $h$.
\item The subgroup $H(L)=(G_1(L_{\text{root}})\rtimes  G_2(L_{\text{root}}))\cap \text{Aut}(L)$ is generated by $h_1h$ and $h_2h$.
\end{enumerate}
\end{prop}

\begin{proof}
Recall, \cite{CS}, that the glue code is generated by 
\[[2,4,0],\,\,\,\,[5,0,1],\,\,\,\,[0,5,3],\]
 and that $G_1(L)\simeq \mathbb Z/2\mathbb Z$ and  $G_2(L)\simeq \mathbb Z/2\mathbb Z$.
\begin{enumerate}
\item
We verify that $\gamma$ belongs to $G_1(L_{\text{root}})$, preserves the glue code and is of order $2$.
\item According to $2.2$, it suffices to exhibit an element $h_1\in G_1(L_{\text{root}})$ such that $h_1h\in (G_1(L_{\text{root}})\rtimes  G_2(L_{\text{root}}))\cap \text{Aut}(L)$, i.e. preserving the glue code of $L$. We verify easily
$h_1h([2,4,0])=[6,2,0]=3\times [2,4,0]$, $h_1h([5,0,1])=[0,5,3]$ and $h_1h([0,5,3])=[5,0,1]$. Thus $h\in \text{Aut}(L)$ and generates $G_2(L)$ since $G_2(L)\simeq \mathbb Z/2\mathbb Z$.
\item This follows from the previous item and the isomorphisms $G_1(L)\simeq \mathbb Z/2\mathbb Z$ and $G_2(L)\simeq \mathbb Z/2\mathbb Z$ \cite{CS}.
\end{enumerate}
\end{proof}

\begin{cor}\label{cor:1}
The two primitive embeddings of $A_5$ in $D_6$, namely $i_1$ and $i_2$, correspond to at most two elliptic fibrations of $X$, non isomorphic by an automorphism of $N(A_9^{(1)}A_9^{(2)}D_6)$.

\end{cor}
\begin{proof}
From the proposition we deduce that the fibration obtained with the embeddings $A_1=a_1$ in $A_9^{(1)}$ and $i_1(A_5)$ in $D_6$ is isomorphic by the automorphism  $h_2h$ to $A_1=a_9$ embedded in $A_9^{(2)}$ and $i_2(A_5)$ in $D_6$. Similarly,  the fibration obtained with the embeddings $A_1=a_1$ in $A_9^{(2)}$ and $i_1(A_5)$ in $D_6$ is isomorphic by the automorphism of $h_1h$ to $A_1=a_9$ embedded in $A_9^{(1)}$ and $i_2(A_5)$ in $D_6$.

\end{proof}

\section{From primitive embeddings to Mordell-Weil lattices}
Let $X$ the $K3$-surface of discriminant $-12$ studied in \cite{BGL}. To each primitive embedding of $A_5\oplus A_1$ in $L_{\text{root}}$ for $L$ Niemeier lattice, corresponds an elliptic fibration of $X$. Define $W=(A_5\oplus A_1)^{\perp_L}$ and $N=(A_5 \oplus A_1)^{\perp_{L_{\text{root}}}}$. First observe that $W_{\text{root}}=N_{\text{root}}$. Then the configuration of singular fibers in the corresponding elliptic fibration is encoded in the trivial lattice $T(X)$ of the elliptic fibration given by
\[T(X)=U\oplus W_{\text{root}}.\]
The torsion group is given by $\overline{W_{\text{root}}}/W_{\text{root}}$.

The Mordell-Weil lattice $MWL(X)$, that is the Mordell-Weil group modulo its torsion subgroup equipped with the height pairing is given by
\[MWL(X)=W/\overline{W_{\text{root}}},\]
where the bar means the primitive closure.
The height pairing of two points $P$ and $Q$ of the Mordell-Weil group is given by the Shioda's formulae
\begin{equation}\label{E:S1}
\langle P,Q \rangle =2+\bar{P}.\bar{O}+\bar{Q}.\bar{O}-\bar{P}.\bar{Q}-\sum_v \text{contr}_v(P,Q)
\end{equation}
and the height of $P$ by
\begin{equation}\label{E:S2}
h(P)=\langle P,P\rangle =4+2\bar{P}.\bar{O}-\sum_v\text{contr}_v(P)
\end{equation}
where $O$ denotes the zero, the bar their associated sections and $v$ runs through the singular fibers. If $\Theta_{v,i}$ is a component of the singular fiber $\Theta_v$ and if $P$  (resp. $Q$) intersects $\Theta_{v,i}$ (resp. $\Theta_{v,j}$), $i<j$, we recall the table of their contributions, Table \ref{table:cont}.

\begin{table}\caption{Contributions for the height pairing} \label{table:cont}
\begin{center}
\begin{tabular}{|l|l|l|l|l|}
\hline
  fiber &  $IV^*$  & $III^*$  &  $I_n\,\,n>1$  & $I_n^*$ \\
\hline
Dynkin diagram& $E_6$    & $E_7$  & $A_{n-1}$  & $D_{n+4}$ \\
\hline
$i=j$ & $4/3$ & $3/2$ & $i(n-i)/n$ & $\begin{cases}
                                      1 & i=1\\
 
                                      1+n/4 & i=2,3
                                      \end{cases}$ \\
\hline

$i<j$  &  $2/3$   &  -- & $i(n-j)/n$ & $\begin{cases}
                                       1/2 & i=1\\
                                       1/2+n/4 & i=2,3
                                      \end{cases}$ \\

\hline

\end{tabular}
\end{center}
\end{table}

Recall that the single components of an $I_n^*$ fiber, $n>0$, are distinguished into the near component $\Theta_1$ which intersects the same double component as the zero component and the far components $\Theta_2$, $\Theta_3$.   
\subsection{Defining sections of our fibrations}
In each class of $W/N$ we choose a representative in order to form either a torsion or an infinite section of the fibration. The section $V$ is defined as
\[V=kF+mO+\omega,\]
$F$ being the generic fiber, $O$ the zero section, $\omega$ a well choosed glue vector in a coset of $W/N$. Since $V$ has to satisfy $V.F=1$, it follows $m=1$. The rational integer $k$ can be obtained from the relation $V.V=-2$, since $\omega . \omega$ is even. Finally the glue vector $\omega$ is choosed so that $V$ cuts each singular fiber in exactly one point. Then we test if the section cuts or not the zero section in order to apply the height formula (\ref{E:S2}). Sections with height $0$ are torsion sections. Moreover we have to determine infinite sections with a height matrix giving the discriminant of the $K3$ surface, that is in our case $12$, according to the formula \cite{ScSh}
\begin{equation}\label{F:disc}
\text{disc}(NS(X))=(-1)^{\text{rank}E(K)}\text{disc}(T(X))\text{disc}(MWL(X))/(\#E(K)_{\text{tors}})^2.
\end{equation}

\section{The elliptic fibration from $L=N(D_6^4)$}
Take the unique, up to $\text{Aut}(L)$, primitive embedding of $A_5\oplus A_1$ in $L$
given by $\phi(A_5\oplus A_1)=(i_1(A_5),d_6,0,0)$. We get

\[(i_1(A_5))^{\perp_{D_6}} =z_6 =2\delta_6 =2[1], \]

\[(A_1)^{\perp_{D_6}}=\langle d_5 \rangle \oplus \langle x_3:=d_5+d_6+2d_4+d_3,d_3,d_2,d_1 \rangle= A_1\oplus D_4,\]

\[N: =((i_1(A_5)\oplus A_1)^{\perp_{L_{\text{root}}}} =(\langle z_6 \rangle , A_1\oplus D_4,D_6,D_6)\]

and $N_{\text{root}}=(0,A_1\oplus D_4,D_6,D_6)$.
Since $\det N=12\times 4^3$, $\det W=12$, it follows
\[W/N\simeq  \left ( \mathbb Z /2 \mathbb Z \right )^3.\]
An elliptic fibration is characterized by its torsion sections, infinite sections and where these sections cut the singular fibers of the fibration. All these data are encoded in $W/N$ and so we shall first compute these groups. 

Observing that $[2]$ and $[3]$ do not belong to $i_1(A_5)^{\perp_{D_6^*}}$, the elements of the glue code (\ref{glue:G}) belonging to $W/N$ are only those beginning by $0$ or $1$, precisely
\[
\begin{matrix}
[0,0,0,0] & [0,1,2,3] & [0,3,1,2] & [0,2,3,1]\\
[1,1,1,1] & [1,0,3,2] & [1,3,2,0] & [1,2,0,3].
\end{matrix}
\]
Among them only those beginning by $0$ belongs to $\overline{W}_{\text{root}}$. Thus torsion sections can be realized only from the glue vectors
\[[0,0,0,0] , [0,1,2,3] , [0,3,1,2] , [0,2,3,1].\]

Moreover we must choose in them elements belonging to $\overline{W}_{\text{root}}$. Since, in the coset $[3]$, $\tilde{\delta}_6$ satisfies
\[2\tilde{\delta}_6 =d_1+2d_2+d_3+2x_3+d_5 \in D_4\oplus A_1\]
and in coset $[2]$,
\[2\bar{\delta}_6 =2d_1+2d_2+d_3+x_3 \in D_4\]
it is possible to write torsion sections from $[0,3,1,2]$, $[0,2,3,1]$ and $[0,0,0,0]$. It remains to find in the coset $[1]$ an element with the same property, that is $\delta-d_3-d_4-d_6$, since

\[2\delta_6 -2d_3-2d_4-2d_6=d_1+2d_2+x_3+d_5\in D_4\oplus A_1.\]

The Mordell-Weil lattice being $W/\overline{W_{\text{root}}}$, the infinite sections can be realized from the classes
\[[1,0,3,2]  ,  [1,2,0,3] , [1,3,2,0]  , [1,1-d_3-d_4-d_6,1,1].\]

The various contributions to the singular fibers can be derived from Table \ref{table:cont}.

\begin{table}\caption{Contributions}\label{table:cont1}
\begin{center}
\begin{tabular}{|l|l|l|l|l|}
\hline
   &  Contr. on     &  Contr. on  &  Contr. on & Contr. on    \\
   &   $D_4$   &  $ A_1$ &  $D_6$  & $D_6$   \\
\hline
$\bar{\delta}_6 \in [2]$ & 1    & 0  &1  & 1  \\
\hline
$\tilde{\delta}_6 \in [3]$ & 1 & 1/2 & 1+1/2 & 1+1/2  \\
\hline
$\delta-d_3-d_4-d_6 \in [1]$ & 1 & 1/2 & 1+1/2 & 1+1/2  \\
\hline
$\delta \in [1]$ & 1 & 0 & 1+1/2 & 1+1/2  \\
\hline
\end{tabular}
\end{center}
\end{table}

Taking in account the different values $\delta_6^2=\tilde{\delta}_6^2=(\delta_6 -d_3-d_4-d_6)^2=-3/2$ and $\bar{\delta}_6^2=-1$, we can draw a table with the various contributions to height for the different sections in Table \ref{table:height0}.

\begin{table}
\caption{Contributions and heights of the sections from $N(D_6^4)$}\label{table:height0}
\begin{center}
\begin{tabular}{|c|lccccc|}

\hline
 &    &  Contr.      &  Contr.   &  Contr.  & Contr. & ht. \\
  &  &   $D_4$   &  $ A_1$ &  $D_6$  & $D_6$  &  \\
\hline
${\scriptstyle Q_1}$ & $0+2F+[0,2,3,1]$ & 1    & 0  &3/2  & 3/2 & 0 \\
${\scriptstyle Q_3}$  & $0+2F+[0,3,1,2]$   & 1 & 1/2 & 3/2 & 1 & 0 \\
${\scriptstyle Q_2}$  & $0+2F+[0,1{\scriptstyle -d_3-d_4-d_6},2,3]$     & 1 & 1/2 & 1 & 3/2 & 0 \\
${\scriptstyle W_1}$ & $0+2F+[1,0,3,2]$    & 0 & 0 & 3/2 & 1 & 3/2 \\
${\scriptstyle W_1+Q_1}$ & $0+2F+[1,2,0,3]$ & 1 & 0 & 0 & 3/2 & 3/2\\
${\scriptstyle W_1+Q_3}$ & $0+2F+[1,3,2,0]$ & 1 & 1/2 & 1 & 0 & 3/2 \\
${\scriptstyle W_1+Q_2}$ & $0+3F+[1,1{\scriptstyle -d_3-d_4-d_6},1,1]$ & 1 & 1/2 & 3/2 & 3/2 & 3/2\\
\hline

\end{tabular}
\end{center}
\end{table}

It is easily derived that the torsion group of the elliptic fibration is isomorphic to $\mathbb Z/2\mathbb Z \times \mathbb Z/2\mathbb Z$ and the Mordell-Weil lattice is generated by a section of height $3/2$, in concordance with the formula (\ref{F:disc}),
\[-12=\text{disc}NS(X)=-4 \times 4\times 4 \times 2 \times \frac{3}{2}\times \frac{1}{4^2}.\]

Thus we have proved the following result.

\begin{prop}

The elliptic fibration on the $K3$-surface $X$ derived from Niemeier lattice $L=N(D_6^4)$ has singular fibers of type $A_1$ ($I_2$), $D_4$ ($I_0^*)$, $D_6$ ($I_2^*$), $D_6$ ($I_2^*$). Its Mordell-Weil group has rank $1$ and torsion part isomorphic to $\mathbb Z/2\mathbb Z \times \mathbb Z/2\mathbb Z$. Its Mordell-Weil lattice is generated by an infinite section of height $3/2$.
\end{prop}

\section{The elliptic fibrations from $L=N(A_9^2D_6)$}
Let $L$ be the Niemeier lattice with $L=N(A_9^{(1)}A_9^{(2)}D_6)$. By \cite{CS} we know that $L$ is obtained from the following glue vectors

\[L/L_{{\text{root}}}=\langle [2,4,0],[5,0,1],[0,5,3] \rangle,\]
where $1$ denotes the coset in $A_9^*/A_9$ of $\alpha=\frac{1}{10}(9a_1+8a_2+7a_3+6a_4+5a_5+4a_6+3a_7+2a_8+a_9)$. 
From Corollary \ref{cor:1} we know that we have at most two elliptic fibrations coming from the Niemeier lattice $L=N(A_9^2D_6)$ non isomorphic by an automorphism of $L$. We shall prove that we have effectively two.

\subsection{First embedding in $D_6$}

We embed $A_1$ in $A_9^{(1)}$ by $\phi(A_1)=\langle a_1^{(1)} \rangle$ and $A_5$ in $D_6$ by $i_1(A_5)=(d_5,d_4,d_3,d_2,d_1)$. As computed in \cite{BGL}, we obtain
\[N=\left ( \phi(A_1)\oplus i_1(A_5)\right ) ^{\perp_{L_{\text{root}}}}=[\langle a_1+2a_2,a_3, ...,a_9 \rangle, A_9,  \langle z_6\rangle ] \]
with $z_6=d_1+2d_2+3d_3+4d_4+2d_5+3d_6$ and $\det(\langle a_1+2a_2,a_3, ...,a_9 \rangle)=2\times 10$; thus $\det(N)=2\times 10\times 10 \times 6$.
It follows $N_{\text{root}}=[\langle a_3,...,a_9\rangle, A_9, 0] \simeq A_7^{(1)}\oplus A_9^{(2)}$ and $W/N=\langle [2,4,0],[5,0,1] \rangle \simeq \mathbb Z/10\mathbb Z$. Since there is no integer $k$ satisfying $k( [2,4,0])\in N_{\text{root}}$ and no integer $k'$ with $k'([5,0,1])\in N_{\text{root}}$, we deduce that $\overline{W_{\text{root}}}/W_{\text{root}}=(0)$ so the corresponding elliptic fibration has trivial torsion and rank $2$.

Now we want to determine the Mordell-Weil lattice of the fibration, in our case
\[\text{MWL}(X)=W/\overline{W_{\text{root}}}\simeq W/W_{\text{root}}.\]

The infinite sections are derived from elements of the glue code of $W/N$, namely from 
\[
\begin{matrix}
[2,4,0],& [4,8,0], & [6,2,0], & [8,6,0], & [0,0,0]\\
[5,0,1], & [7,4,1], & [9,8,1], & [1,2,1], &[3,6,1].
\end{matrix}
\]

We define sections as explained in 5.1 so we search in each coset $j$ an element $\alpha_j$ satisfying $\alpha_j.a_j=1$ and $\alpha_j.a_i=0$. We obtain a unique solution 
\[-\alpha_j:=j\alpha-(j-1)a_1-(j-2)a_2...-a_{j-1}.\]
We observe that $\alpha_j \in W$ for all $j$ but $j=1$. Thus we choose in the coset of $\alpha_1$ an element in $W$ and cutting $A_7=\langle a_3,a_4, ..., a_9\rangle$ in exactly one point, namely $-\bar{\alpha_1}=\alpha -a_1-a_2$. The elements $(\alpha_1, \alpha_2, ..., \alpha_9)$ are in fact the dual elements $(a_1^*, a_2^*, ..., a_9^*)$. So their Gram matrix is minus the inverse matrix of the Gram matrix of the $a_i$, namely

\setlength\arraycolsep{2pt}

\renewcommand{\arraystretch}{1.4}

\begin{equation}\label{matrix:1}
\begin{pmatrix}
\frac{9}{10} & \frac{4}{5} & \frac{7}{10} & \frac{3}{5} & \frac{1}{2} &  \frac{2}{5} & \frac{3}{10} & \frac{1}{5} & \frac{1}{10} \\
\frac{4}{5} & \frac{8}{5} & \frac{7}{5} & \frac{6}{5} & {\scriptstyle 1}  &  \frac{4}{5} & \frac{3}{5} & \frac{2}{5} & \frac{1}{5} \\
\frac{7}{10} & \frac{7}{5} & \frac{21}{10} & \frac{9}{5} & \frac{3}{2} &  \frac{6}{5} & \frac{9}{10} & \frac{3}{5} & \frac{3}{10} \\
\frac{3}{5} & \frac{6}{5} & \frac{9}{5} & \frac{12}{5} & {\scriptstyle 2}  &  \frac{8}{5} & \frac{6}{5} & \frac{4}{5} & \frac{2}{5} \\
\frac{1}{2} & {\scriptstyle 1}  & \frac{3}{2} & {\scriptstyle 2}  & \frac{5}{2} & {\scriptstyle 2}  & \frac{3}{2} & {\scriptstyle 1}  & \frac{1}{2} \\
\frac{2}{5} & \frac{4}{5} & \frac{6}{5} & \frac{8}{5} &{\scriptstyle 2}  &  \frac{12}{5} & \frac{9}{5} & \frac{6}{5} & \frac{3}{5} \\
\frac{3}{10} & \frac{3}{5} & \frac{9}{10} & \frac{6}{5} & \frac{3}{2} &  \frac{9}{5} & \frac{21}{10} & \frac{7}{5} & \frac{7}{10} \\
\frac{1}{5} & \frac{2}{5} & \frac{3}{5} & \frac{4}{5} & {\scriptstyle 1}  &  \frac{6}{5} & \frac{7}{5} & \frac{8}{5} & \frac{4}{5} \\
\frac{1}{10} & \frac{1}{5} & \frac{3}{10} & \frac{2}{5} & \frac{1}{2} &  \frac{3}{5} & \frac{7}{10} & \frac{4}{5} & \frac{9}{10} 
\end{pmatrix}
\end{equation}

We read directly on the above matrix
\[\alpha_1^2=\alpha_9^2=-\frac{9}{10},\alpha_2^2=\alpha_8^2=-\frac{8}{5},\alpha_3^2=\alpha_7^2=-\frac{2
1}{10},\alpha_4^2=\alpha_6^2=-\frac{12}{5},\alpha_5^2=-\frac{5}{2}\]
and we compute $\bar{\alpha_1}^2=-\frac{9}{10}$.

Hence we obtain the nine non zero sections $V_i$, $1\leq i \leq 9$, quoted in the  Table \ref{table:height1}. Using the entries of the matrix (\ref{matrix:1}) we obtain their contributions to the singular fibers, their heights and the various $\langle V_i,V_1 \rangle$ and  $\langle V_i,V_2 \rangle$, according to formulae (\ref{E:S1}) and (\ref{E:S2}). Moreover the determinant of the height matrix of $V_1,V_2$ is equal to $\frac{3}{20}$ fitting with the formula (\ref{F:disc}). These data allow in turn to express $V_j$ for $j\geq 3$ as a linear combination of $V_1$ and $V_2$.
For example, looking for a relation $V_3=aV_1+bV_2$, we compute $\langle V_3, V_k \rangle=a \langle V_1, V_k \rangle +b\langle V_2, V_k \rangle$ with $k=1,2$. Thus we get two equations in $a,b$ and solving  the system it follows $a=b=1$.

Finally the order in the Table \ref{table:height1} refers to the order of the element in $W/N$ of the corresponding section. 

\begin{table}

\begin{center}
\caption{Height and pairing-First embedding} \label{table:height1}
\begin{tabular}
[c]{lllllllll}%
&  & $\scriptstyle{I_{8}}$ & $\scriptstyle{I_{10}}$ & ${\scriptstyle <V_{i},V_{1}>}$ &
${\scriptstyle <V_{i},V_{2}>}$ & ${\scriptstyle ht\left(  V_{i}\right)}  $ &
${\scriptstyle order}$ & \\
$V_{1}$ & ${\scriptstyle O+2F+\left[  \alpha_{9},\alpha_{8},1\right]  }$ & $7$ & $8$ &
$\frac{61}{40}$ & $\frac{1}{20}$ & $\frac{61}{40}$ & ${\scriptstyle 10}$ & $V_1$\\
$V_{2}$ & $ {\scriptstyle O+2F+\left[  \alpha_{8},\alpha_{6},0\right] } $ & $6$ & $6$ &
$\frac{1}{20}$ & $\frac{1}{10}$ & $\frac{1}{10}$ & ${\scriptstyle 5}$ &$V_2$ \\
$V_{3}$ & $ {\scriptstyle O+3F+\left[  \alpha_{7},\alpha_{4},1\right] } $ & $5$ & $4$ &
$\frac{63}{40}$ & $\frac{3}{20}$ & $\frac{69}{40}$ & ${\scriptstyle 10}$ & $V_{1}+V_{2}$\\
$V_{4}$ & $ {\scriptstyle O+2F+\left[  \alpha_{6},\alpha_{2},0\right] } $ & $4$ & $2$ &
$\frac{1}{10}$ & $\frac{1}{5}$ & $\frac{4}{10}$ & ${\scriptstyle 5}$ & $2V_{2}$\\
$V_{5}$ & ${\scriptstyle O+2F+\left[  \alpha_{5},0,1\right] } $ & $3$ & $0$ & $\frac{13}{8}$%
& $\frac{1}{4}$ & $\frac{17}{8}$ & ${\scriptstyle 2}$ & $V_{1}+2V_{2}$\\
$V_{6}$ & ${\scriptstyle O+2F+\left[  \alpha_{4},\alpha_{8},0\right] } $ & $2$ & $8$ &
$\frac{3}{20}$ & $\frac{3}{10}$ & $\frac{9}{10}$ & ${\scriptstyle 5}$ & $3V_{2}$\\
$V_{7}$ & $ {\scriptstyle O+3F+\left[  \alpha_{3},\alpha_{6},1\right] } $ & $1$ & $6$ &
$\frac{67}{40}$ & $\frac{7}{20}$ & $\frac{109}{40}$ & ${\scriptstyle 10}$ & $V_{1}+3V_{2}$\\
$V_{8}$ & $ {\scriptstyle O+2F+\left[  \alpha_{2},\alpha_{4},0\right] } $ & $0$ & $4$ &
$\frac{1}{5}$ & $\frac{2}{5}$ & $\frac{8}{5}$ & ${\scriptstyle 5}$ & $4V_{2}$\\
$V_{9}$ & $ {\scriptstyle O+2F+\left[ \alpha_{1}-a_{1}-a_{2},\alpha_{2},1 \right]}  $ & $1$ & $2$ & $\frac{59}{40}$ &
$\frac{-1}{20}$ & $\frac{61}{40}$ & ${\scriptstyle 10}$ & $V_{1}-V_{2}$\\

$V_{11}$ & ${\scriptstyle O+2F+\left[  -a_1-2a_2-a_{3},0,0\right]  }$ & $2$ & $ 0$ & $\frac{-1}%
{4}$ & $\frac{-1}{2}$ & $\frac{5}{2}$ & ${\scriptstyle 0}$ & $-5V_{2}$\\
$V_{12}$ & ${\scriptstyle O+3F+\left[  0,0,2\delta_{6}\right]  }$ & $0$ & $0$ & $3$ & $0$ &
$6$ &${\scriptstyle 0}$  & $2V_{1}-V_{2}$%
\end{tabular}
\end{center}
\end{table}

\begin{theo}
The Mordell-Weil lattice can be generated by the section $V_{2}$ and another
section whose class in $W/N$ is of order $10$ or $2$ ($V_{1},$ $V_{3}%
,V_{7},V_{9}$ or $V_{5}$). It also can be generated by $V_1$ and $V_3$ or $V_9$. 

The rational quadratic forms associated to these various height matrices are all equivalent to the quadratic form $Q(x,y)=\frac{1}{40}(61x^2+4xy+4y^2)$.

The sublattice of index $10$, $N/N_{root}$ of $W/N_{root}$, is generated by $V_{11}=-5V_{2}$ and $V_{12}$ with 
$<V_{11},V_{12}>=0.$
\end{theo}
\begin{proof}
We observe that the nine first sections are not in the same class modulo $N_{\text{root}}$.

The rational quadratic form $Q(x,y)$ associated to the height matrix of ($V_1$, $V_2$) is $Q(x,y)=\frac{1}{40}(61x^2+4xy+4y^2)$.


Other properties  are simple transcriptions of the base change which can be derived from the last column of Table \ref{table:height1}.
For example the rational quadratic form associated to the height matrix of ($V_1$, $V_9$) is equivalent to $Q(x,y)$ since 
\[
\begin{pmatrix}
1 & 0\\
1 & -1
\end{pmatrix}
\begin{pmatrix}
\frac{61}{40} & \frac{1}{20}\\
\frac{1}{20} & \frac{1}{10}
\end{pmatrix}
\begin{pmatrix}
1 & 1\\
0 & -1
\end{pmatrix}
=
\begin{pmatrix}
\frac{61}{40} & \frac{59}{40}\\
\frac{59}{40} & \frac{61}{40}
\end{pmatrix}
\]

Finally we verify that the height matrix of $(V_{11},V_{12})$, namely $\begin{pmatrix}
                                                                 \frac{5}{2} & 0 \\
                                                                  0 & 6
                                                                \end{pmatrix}$,     has  determinant $15=10^2 \frac{3}{20}.$
Moreover there exists a sublattice of index $2$ generated by $V_{2}$ and $V_{12}$ with
$<V_{2},V_{12}>=0.$

\end{proof}

\subsection{Second embedding in $D_6$}

We embed $A_1$ in $A_9^{(1)}$ by $\phi(A_1)=\langle a_1^{(1)} \rangle$ and $A_5$ in $D_6$ by $i_2(A_5)=(d_6,d_4,d_3,d_2,d_1)$. We obtain
\[N=\left ( \phi(A_1)\oplus i_1(A_5)\right ) ^{\perp_{L_{\text{root}}}}=[\langle a_1+2a_2,a_3, ...,a_9 \rangle, A_9,  \langle\tilde{ z_6}\rangle ] \]
with $\tilde{z_6}=d_1+2d_2+3d_3+4d_4+3d_5+2d_6$ and $\det(\langle a_1+2a_2,a_3, ...,a_9 \rangle)=2\times 10$; thus $\det(N)=2\times 10\times 10 \times 6$.
It follows $N_{\text{root}}=[\langle a_3,...,a_9\rangle, A_9, 0] \simeq A_7^{(1)}\oplus A_9^{(2)}$ and $W/N=\langle [2,4,0],[0,5,3] \rangle \simeq \mathbb Z/10\mathbb Z$. Since there is no integer $k$ satisfying $k( [2,4,0])\in N_{\text{root}}$ and no integer $k'$ with $k'([0,5,3])\in N_{\text{root}}$, we deduce that $\overline{W_{\text{root}}}/W_{\text{root}}=(0)$ so the corresponding elliptic fibration has trivial torsion and rank $2$.

\begin{table}
\begin{center}
\caption{Height and pairing-Second embedding} \label{table:height2}
\begin{tabular}
{lllllllll}%
&  & ${\scriptstyle I_{8}}$ & ${\scriptstyle I_{10}}$ & ${\scriptstyle <Z_{1},Z_{i}>}$ & ${\scriptstyle <Z_{2},Z_{i}>}$ & ${\scriptstyle ht\left(Z_{i}\right)}  $ & ${\scriptstyle o.}$ & \\
$Z_{1}$ & ${\scriptstyle O+3F+\left[  \alpha_{4},\alpha_{3},3\right]}  $ & ${\scriptstyle 2}$ & ${\scriptstyle 3}$ &
$\frac{12}{5}$ & $\frac{3}{10}$ & $\frac{12}{5}$ & ${\scriptstyle 10}$ & $Z_{1}$\\
$Z_{2}{\scriptstyle }$ & ${\scriptstyle O+2F+\left[  \alpha_{8},\alpha_{6},0\right]}  $ & ${\scriptstyle 6}$ & ${\scriptstyle 6}$ &
$\frac{3}{10}$ & $\frac{1}{10}$ & $\frac{1}{10}$ & ${\scriptstyle 5}$ & $Z_{2}$\\
$Z_{3}$ & ${\scriptstyle O+2F+\left[  \alpha_{2},\alpha_{9},3\right] } $ & ${\scriptstyle 0}$ & ${\scriptstyle 9}$ &
$\frac{27}{10}$ & $\frac{2}{5}$ & $\frac{31}{10}$ & ${\scriptstyle 10}$ & $Z_{1}+Z_{2}$\\
$Z_{4}{\scriptstyle }$ & ${\scriptstyle O+2F+\left[  \alpha_{6},\alpha_{2},0\right] } $ & ${\scriptstyle 4}$ & ${\scriptstyle 2}$ &
$\frac{3}{5}$ & $\frac{1}{5}$ & $\frac{2}{5}$ & ${\scriptstyle 5}$ & $2Z_{2}$\\
$Z_{5}$ & ${\scriptstyle O+2F+\left[  0,\alpha_{5},3\right] } $ & ${\scriptstyle 0}$ & ${\scriptstyle 5}$ & $\frac{3}{2}$ &
${\scriptstyle 0}$ & $\frac{3}{2}$ & ${\scriptstyle 2}$ & $Z_{1}-3Z_{2}$\\
$Z_{6}{\scriptstyle }$ & ${\scriptstyle O+2F+\left[  \alpha_{4},\alpha_{8},0\right] } $ & ${\scriptstyle 2}$ & ${\scriptstyle 8}$ &
$\frac{9}{10}$ & $\frac{3}{10}$ & $\frac{9}{10}$ & ${\scriptstyle 5}$ & $3Z_{2}$\\
$Z_{7}$ & ${\scriptstyle O+2F+\left[  \alpha_{8},\alpha_{1},3\right] } $ & ${\scriptstyle 6}$ & ${\scriptstyle 1}$ &
$\frac{9}{5}$ & $\frac{1}{10}$ & $\frac{8}{5}$ & ${\scriptstyle 10}$ & $Z_{1}-2Z_{2}$\\
$Z_{8}{\scriptstyle }$ & ${\scriptstyle O+2F+\left[  \alpha_{2},\alpha_{4},0\right] } $ & ${\scriptstyle 0}$ & ${\scriptstyle 4}$ &
$\frac{6}{5}$ & $\frac{4}{10}$ & $\frac{8}{5}$ & ${\scriptstyle 5}$ & $4Z_{2}$\\
$Z_{9}$ & ${\scriptstyle O+3F+\left[  \alpha_{6},\alpha_{7},3\right] } $ & ${\scriptstyle 4}$ & ${\scriptstyle 7}$ &
$\frac{21}{10}$ & $\frac{1}{5}$ & $\frac{19}{10}$ & ${\scriptstyle 10}$ & $Z_{1}-Z_{2}$\\
$Z_{11}$ & ${\scriptstyle O+2F+\left[-a_1-2a_2-a_{3},0,0\right] } $ & ${\scriptstyle 2}$ & ${\scriptstyle 0}$ & $\frac{-1}%
{4}$ & $\frac{-1}{2}$ & $\frac{5}{2}$ & ${\scriptstyle 0}$ & $-5Z_{2}$\\
$Z_{12}$ & ${\scriptstyle O+3F+\left[  0,0,\tilde{\delta}_{6}\right]  }$ & ${\scriptstyle 0}$ & ${\scriptstyle 0}$ & ${\scriptstyle 3}$ &
${\scriptstyle 0}$ & ${\scriptstyle 6}$ & ${\scriptstyle 0}$ & ${\scriptstyle 2Z_{1}-6Z_2}$\\
\end{tabular}
\end{center}
\end{table}

\begin{theo}
The Mordell-Weil lattice can be generated by the section $Z_{2}$ and another
section whose class in $W/N$ is of order $10$ or $2$ ($Z_{1},$ $Z_{3}%
,Z_{7},Z_{9}$ or $Z_{5}$). It also can be generated by $Z_1$ and $Z_3$ or $Z_9$. The rational quadratic forms associated to these various height matrices are all equivalent to the quadratic form $\frac{1}{10}(x^2+15y^2)$.

The sublattice of index $10$, $N/N_{root}$, is generated by $Z_{11}=-5Z_{2}$ and $Z_{12}.$

\end{theo}

\begin{proof}
The proof is similar to the previous proof.

\end{proof}











\begin{cor}
The Mordell-Weil lattices for the first $i_1$ and second $i_2$ embeddings are not isomorphic. Thus they lead to two distinct elliptic fibrations.
\end{cor}

\begin{proof}
According to the previous theorems, the Mordell-Weil lattice for the first (resp. second) embedding can be generated by the sections $V_1$ and $V_2$ (resp. $Z_2$ and $Z_5$) with height matrix $\begin{pmatrix}
\frac{61}{40}  &  \frac{1}{20} \\
\frac{1}{20}  &  \frac{1}{10}
\end{pmatrix}$
(resp.  $\begin{pmatrix}
\frac{1}{10}  &  0 \\
0  &  \frac{3}{2}
\end{pmatrix}$).

As we can prove easily that these two matrices are not equivalent, since there is no matrix $\begin{pmatrix}
                                                                                             a & b \\
                                                                                
                                                                                             c & d
                                                                                             \end{pmatrix}$  
with integer entries satisfying
\[
\begin{pmatrix}
a & b\\
c & d
\end{pmatrix}
\begin{pmatrix}
\frac{1}{10} & 0 \\
0 & \frac{3}{2}
\end{pmatrix}
\begin{pmatrix}
a & c\\
b & d
\end{pmatrix}
=
\begin{pmatrix}
\frac{a^2}{10}+\frac{3}{2}b^2 & \frac{ac}{10}+\frac{3}{2}bd\\
 \frac{ac}{10}+\frac{3}{2}bd & \frac{c^2}{10}+\frac{3}{2}d^2
\end{pmatrix}
=
\begin{pmatrix}
\frac{61}{40} & \frac{1}{20}\\
\frac{1}{20} & \frac{1}{10}
\end{pmatrix},
\]

for there are no integers $a$ and $b$ satisfying $4(a^2+15b^2)=61$.

\end{proof}

\newpage

\section{Weierstrass equations}
In this second part we obtain the Weierstrass equations of the unique, up to automorphism of the Niemeier lattice $N(D_6^4)$, elliptic fibration denoted $\#36$ as in \cite{BGL} and of the two elliptic fibrations, non isomorphic by an automorphism of the Niemeier lattice $N(A_9^2D_6)$, numbered $\#40$ as in \cite{BGL} and $\#40$ bis. These fibrations are given with their torsion and infinite sections and their Mordell-Weil lattices so we can easily see the parallelism between the theoretic results of the first part and the new ones coming from the Weierstrass equations. 
\subsection{Background and method}
\setlength\arraycolsep{2pt}
\renewcommand{\arraystretch}{1.4}

We start from fibration \#50 of (\cite{BGL}) with Weierstrass equation
\begin{equation}
E_{u}:y^{2}+(u^{2}+3)yx+(u^{2}-1)^{2}y=x^{3},\label{Eq:Eu}%
\end{equation}
which is the universal elliptic curve with torsion structure $\left(
\mathbb{Z}/2\mathbb{Z}\right)  ^{2}\times\mathbb{Z}/3\mathbb{Z.}$

The points $A_2=(-\frac{1}{4}(u^2-1)^2,0)$, $A_{22}=(-(u+1)^2,(u+1)^3) $ and  $A_{23}=(-(u-1)^2,(u-1)^3) $ are  2-torsion points and 
the point $P_3=(0,0)$ is a 3-torsion point.

The singular fibers are of type $I_6$ for $u=1,-1,\infty$ and $I_2$ for $3,-3,0$. 

The components of an $I_n$ fiber are numbered cyclically, $\Theta_{i,j}$ being the $j-$th component of the singular fiber above $u=i$ and the component $\Theta_{i,0}$
intersecting the zero section.
\subsection{The graph $\Gamma$}
The vertices of the graph $\Gamma$ are the twelve torsion sections and the $24$ components $\Theta_{i,j}$. Two vertices are linked by an edge if they intersect.  To make it easely lisible, only some parts of this graph are drawn on the following figures. 

Recall first that two torsion sections do not intersect.

Then, we compute for a set of generating sections, which component of singular fibers are intersected, derived for example from the method given in \cite {Cr} or in \cite {Si}.  For the other torsion sections, we use the algebraic structure of the N\'eron model or the height of sections as explained below.

Recall that the height of a torsion point $P$ is $0$  involving conditions on $contr_{v}\left(
P\right)  $ since from formula (\ref{E:S2}) and Table \ref{table:cont} it follows $4=\sum_{v}contr_{v}\left(
P\right)$. For example, the only possible sum of contributions for the $3$-torsion point $P_3$ is
$\frac{2\times4}{6}+\frac{2\times4}{6}+\frac{2\times4}{6}+0+0+0,$ and for a
two-torsion point $\frac{3\times3}{6}+\frac{3\times3}{6}+0+\frac{1}{2}%
+\frac{1}{2}+0$. Since the sum of two $2-$torsion points is also a $2-$torsion point,
only one $2-$torsion point intersects the component $\Theta_{i,0}$, for a
given reducible fiber. These remarks allow us  
to construct $\Gamma.$

Let us now summarize useful results. The point $P_{3}$ intersects the component
$\Theta_{i,2}$ (by convention $\Theta_{i,2}$ not $\Theta_{i,4}$ ) of the
$I_{6}$ fibers and the component $\Theta_{i,0}$ of the $I_{2}$ fibers. The
point $A_{2}$ intersects the components $\Theta_{\infty,0}$ and $\Theta
_{0,0},$ the point $A_{22}$ intersects the components $\Theta_{1,0}$ and
$\Theta_{-3,0}$ . These two points intersects $\Theta_{i,3}$ for the others
$I_{6}$ fibers  and $\Theta_{i,1}$ for the other $I_{2}$ fibers.

\subsection{Method for building elliptic fibrations from fibration \#50}
Recall that it is sufficient to identify a divisor $D$ on the surface that has the shape of a singular fiber from Kodaira's list and an irreducible curve $C$ with $C.D=1$ to find an elliptic fibration with $D$ as a singular fiber and $C$ as a section. The fibration is induced by the linear system $|D|$.

Moreover, if we can draw two divisors $D$ and $D'$ on the graph $\Gamma$ with $D.D'=0$ it is easier to determine a new fibration.
We must define a function, called elliptic parameter, with divisor $D'-D.$ Moreover if $D$ and $D'$ are subgraph of $\Gamma$ 
we use the elliptic curve  $E_u$. The method and computations are explicited for the fibration \#36.

\section{Fibration \#36}
\begin{figure}
\begin{center}
\begin{tikzpicture}[scale=1]
\unitlength 1cm
\draw [fill=green](0,0) circle  (0.07cm)node [above] {$0$};
\draw [dashed] (0,0) circle (0.15cm);
\draw [double,fill=black] (0,3) circle (0.07cm) node [above] {$P_3$};

\draw [double,fill=black](0,-7) circle (0.07cm) node [above] {$2P_3$};

\draw [fill=blue](0.5,5) circle (0.07cm) node [above] {$A_2$};

\draw [fill=pink] (0,-2)circle (0.07cm) node[left]{$\Theta_{\infty,0}$};
\draw [double,fill=black] (0,-5)circle (0.07cm) node[left]{$\Theta_{\infty,3}$};
\draw [double,fill=black] (1,-3)circle (0.07cm) node[left]{$\Theta_{\infty,5}$};
\draw [double,fill=black] (1,-4)circle (0.07cm) node[left]{$\Theta_{\infty,4}$};
\draw [double,fill=black] (-1,-4)circle (0.07cm) node[left]{$\Theta_{\infty,2}$};
\draw [double,fill=black] (-1,-3)circle (0.07cm) node[left]{$\Theta_{\infty,1}$};
\draw (0,-2)--(1,-3);
\draw (1,-3)--(1,-4);
\draw (1,-4)--(0,-5);
\draw  (0,-5)--(-1,-4);
\draw (-1,-4)--(-1,-3);
\draw (-1,-3)--(0,-2);
\draw [fill=green] (1.732,1)circle (0.07cm) node[left]{$\Theta_{-1,0}$};
\draw [dashed] (1.732,1) circle (0.15cm);
\draw [fill=blue] (4.33,2.5)circle (0.07cm) node[right]{$\Theta_{-1,3}$};
\draw [fill=green] (2.098,2.366)circle (0.07cm) node[left]{$\Theta_{-1,1}$};
\draw [dashed] (2.098,2.366) circle (0.15cm);
\draw [fill=red] (2.9641,2.866)circle (0.07cm) node[left]{$\Theta_{-1,2}$};
\draw [fill=red] (3.964,1.1339)circle (0.07cm) node[right]{$\Theta_{-1,4}$};
\draw [fill=green] (3.098,0.6339)circle (0.07cm) node[left]{$\Theta_{-1,5}$};
\draw [dashed] (3.098,0.6339) circle (0.15cm);
\draw (1.732,1)--(2.098,2.366);
\draw (2.098,2.366)--(2.9641,2.866);
\draw (2.9641,2.866)--(4.33,2.5);
\draw (4.33,2.5)--(3.964,1.1339);
\draw (3.964,1.1339)--(3.098,0.6339);
\draw (3.098,0.6339)--(1.732,1);

\draw [fill=green] (-1.732,1)circle (0.07cm) node[right]{$\Theta_{1,0}$};
\draw [dashed] (-1.732,1) circle (0.15cm);
\draw [fill=blue] (-4.33,2.5)circle (0.07cm) node[left]{$\Theta_{1,3}$};
\draw [fill=green] (-2.098,2.366)circle (0.07cm) node[right]{$\Theta_{1,1}$};
\draw [dashed] (-2.098,2.366) circle (0.15cm);
\draw [fill=red] (-2.9641,2.866)circle (0.07cm) node[left]{$\Theta_{1,2}$};
\draw [fill=red] (-3.964,1.1339)circle (0.07cm) node[left]{$\Theta_{1,4}$};
\draw [fill=green] (-3.098,0.6339)circle (0.07cm) node[left]{$\Theta_{1,5}$};
\draw [dashed] (-3.098,0.6339) circle (0.15cm);

\draw (-1.732,1)--(-2.098,2.366);
\draw (-2.098,2.366)--(-2.9641,2.866);
\draw (-2.9641,2.866)--(-4.33,2.5);
\draw (-4.33,2.5)--(-3.964,1.1339);
\draw (-3.964,1.1339)--(-3.098,0.6339);
\draw (-3.098,0.6339)--(-1.732,1);

\draw (0,0)--(0,-2);
\draw (0,0)--(1.732,1);
\draw (0,0)--(-1.732,1);

\draw (0,-7)--(1,-4);
\draw (0,-7)--(-3.964,1.1339);
\draw (0,-7)--(3.964,1.1339);

\draw (0,3)--(-1,-4);
\draw (0,3)--(-2.9641,2.866);
\draw (0,3)--(2.9641,2.866);

\draw (0.5,5)--(4.33,2.5);
\draw (0.5,5)--(-4.33,2.5);
\draw (0.5,5)--(0,-2);

\draw (-2.2,-1)--(-1.2,-1);
\draw (-2.2,-0.95)--(-1.2,-0.95);
\draw [fill=pink] (-2.2,-1)circle (0.07cm) node[above]{$\Theta_{-3,0}$};
\draw [fill=blue] (-1.2,-1)circle (0.07cm) node[above]{$\Theta_{-3,1}$};
\draw (-1.2,-1)--(0.5,5);

\draw (1.15,-0.7)--(2.15,-0.7);
\draw (1.15,-0.65)--(2.15,-0.65);
\draw [fill=pink] (1.15,-0.7)circle (0.07cm) node[above]{$\Theta_{3,0}$};
\draw [fill=blue] (2.15,-0.7)circle (0.07cm) node[above]{$\Theta_{3,1}$};
\draw (2.15,-0.7)--(0.5,5);

\end{tikzpicture}%
\caption{Fibration $\#36$}
\label{fig:Grw}
\end{center}
\end{figure}
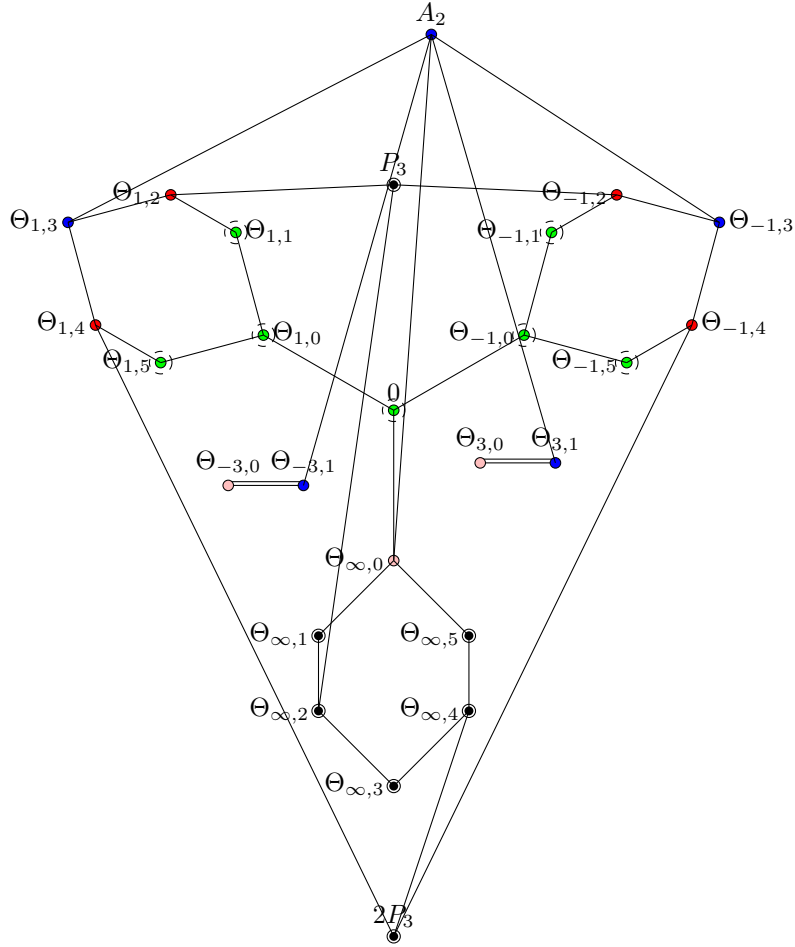

\subsection{Weierstrass equation}
We consider the divisors drawn in black (double circle) for $D'$ and green (dashed circle) for $D$ on the graph (Figure \ref{fig:Grw}) namely
\begin{align*}
D=\Theta_{-1,1}+2\Theta_{-1,0}+\Theta_{-1,5}+2(0)+2\Theta_{1,0}+\Theta_{1,1}+\Theta_{1,5}\\
D'=(P_3)+\Theta_{\infty,5}+2\Theta_{\infty,4}+2\Theta_{\infty,3}+2\Theta_{\infty,2}+\Theta_{\infty,1}+(2P_3).
\end{align*}

The divisors $D$ and $D'$ correspond to two singular fibers of type $I_2^*$ of the same fibration since $D.D'=0$.

We see also that $\Theta_{-1,3},\Theta_{1,3}$ and $A_2$ in blue are a part of another singular fiber.

Let $w$ be a parameter for the new fibration such that $w=\infty$ on $D$ and $0$ on $D'$.

So the divisors $D$ and $D^{\prime}$ correspond to the same element in the
N\'{e}ron-Severi group $NS(X)$. Let $D=\delta+\Delta$ and $D^{\prime}%
=\delta^{\prime}+\Delta^{\prime}$ where $\delta,\delta^{\prime}$ are sums
of sections, $\delta=2(O)$ and $\delta^{\prime}=(P_{3})+(2P_{3}),$ while $\Delta,$
$\Delta^{\prime}$ are sums of components of singular fibers. It follows from the equality $\delta=\delta^{\prime}$ in the group $NS(X)/T(X)$ that  $\delta
-\delta^{\prime}=2(0)-(P_3)-(2P_3)$ is the divisor of a function  on the elliptic curve
$E_{u},$ precisely the function
$x.$ The parameter $w$ is then equal to $x.f\left(  u\right)  .$ We compute
$f\left(  u\right)  $ using three blow- up  to get a pole of order $1$ on
$\Theta_{1,1},\Theta_{-1,5},\Theta_{\infty,1}$ and obtain
\[
w=\frac{x}{\left(  u^{2}-1\right)  ^{2}}.%
\]
Eliminating $x$ in the equation of $E_u$ and setting $y=(u^2-1)^2z,u=1+U$ it follows a quartic equation in $z,U,w$.
All the transformations are summarized in the birational transformation $\phi:(X,Y,w) \mapsto (x,y,u)$ leading to the following Weierstrass
equation $E_{w}$

\begin{equation}
E_{w}:Y^{2}=X\left(  X-w\left(  1+4w\right)  \right)  \left(  X+w^{2}\left(
1+4w\right)  \right)  \label{Eq:Ew}%
\end{equation}

with 
\begin{align*}
x  &  =\frac{w\left(  1+4w\right)  ^{2}\left(  X+4w^{3}\right)  ^{2}\left(
2Y+X\left(  2w+1\right)  \right)  ^{2}}{\left(  Y-Xw-2w^{3}\left(
1+4w\right)  \right)  ^{4}}\\
 y  &  =-\frac{\left(  1+4w\right)  ^{3}X\left(  X+4w^{3}\right)
^{4}\left(  2Y+X\left(  2w+1\right)  \right)  ^{2}}{\left(  Y-Xw-2w^{3}\left(
1+4w\right)  \right)  ^{6}}\\
u & =\frac{\left(1+4w\right)
\left(  X+4w^3\right)  }{Y-Xw-2w^{3}\left(  1+4w\right)  }+1.
\end{align*}%

The singular fibers are  of type $I_2^*$ for $w=0,\infty$, $I_0^*$ for $w=-1/4$,
$I_2$ for $w=-1.$

We compute that the function $w+1/4$ is equal to $0$ on $\Theta_{\pm 3,1}$, giving thus with $A_{2}$ and $\Theta_{\pm1,3}$ a complete description of the singular fiber $I_{0}^{\ast}.$ 

The component $\Theta_{0,1}$ is a component of the singular fiber $I_2$ obtained for $w=-1$ and does not intersect the new $0$ section. The second component is the curve with the parametrization

\[
u=-2\,{\frac {-3+{z}^{2}}{3+{z}^{2}}}
\]
\[
x=-9\,{\frac {\left (z-1\right )^{2}\left (3+z\right )^{2}\left (z-3
\right )^{2}\left (z+1\right )^{2}}{\left (3+{z}^{2}\right )^{4}}} \,
y=27\,{\frac {\left (3+z\right )^{2}\left (z-1\right )^{2}\left (z+1
\right )^{4}\left (z-3\right )^{4}}{\left (3+{z}^{2}\right )^{6}}}.
\] 

This component gives a quadratic section on $E_u$ and can be used to construct other fibrations.

\subsection {Sections of the fibration \#36}
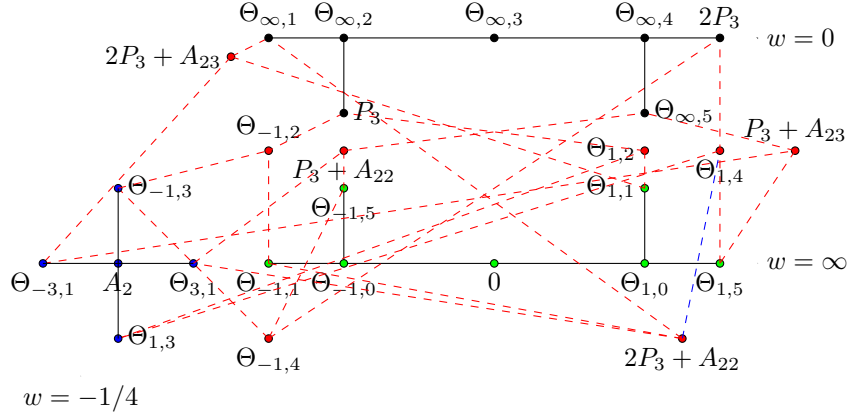
\begin{figure}
\begin{center}
\begin{tikzpicture}[scale=1]%
\draw (-3,-2)--(3,-2);
\draw (-2,-1)--(-2,-2);
\draw (2,-1)--(2,-2);
\draw [fill=green](-3,-2) circle (0.05cm) node [below]{$\Theta_{-1,1}$};
\draw [fill=green](-2,-2) circle (0.05cm) node [below]{$\Theta_{-1,0}$};
\draw [fill=green](0,-2) circle (0.05cm)node [below]{$0$};
\draw [fill=green](2,-2) circle (0.05cm)node [below]{$\Theta_{1,0}$};
\draw [fill=green](3,-2) circle (0.05cm)node [below]{$\Theta_{1,5}$};
\draw (3.5,-2) circle (0.0001cm) node [right] {$w=\infty$};
\draw [fill=green](-2,-1) circle (0.05cm)node [below]{$\Theta_{-1,5}$};
\draw [fill=red] (-3,-3) circle (0.05cm) node [below] {$\Theta_{-1,4}$};
\draw [fill=green](2,-1) circle (0.05cm)node [left]{$\Theta_{1,1}$};
\draw [fill=red](-3,-0.5) circle (0.05cm)node [above]{$\Theta_{-1,2}$};
\draw [fill=red](3,-0.5) circle (0.05cm)node [below]{$\Theta_{1,4}$};
\draw [fill=red](2,-0.5) circle (0.05cm)node [left]{$\Theta_{1,2}$};
\draw [fill=red](-2,-0.5) circle (0.05cm)node [below]{$P_3+A_{22}$};
\draw [fill=red](2.5,-3) circle (0.05cm) node [below] {$2P_3+A_{22}$};
\draw [fill=red](-3.5,0.75) circle (0.05cm)node [left]{$2P_3+A_{23}$};
\draw [fill=red](4,-0.5) circle (0.05cm) node [above]{$P_3+A_{23}$};

\draw [dashed,red] (-6,-2)--(4,-0.5);
\draw [dashed,red] (3,-0.5)--(-5,-3);
\draw [dashed,red] (-5,-3)--(2,-0.5);
\draw [dashed,red] (2,-0.5)--(-2,0);
\draw [dashed,red] (-2,0)--(-3,-0.5);
\draw [dashed,red] (-3.5,0.75)--(-6,-2);
\draw [dashed,red] (2.5,-3)--(-3,1);
\draw [dashed,blue] (2.5,-3)--(3,-0.5);
\draw [dashed,red] (2.5,-3)--(-4,-2);
\draw [dashed,red] (2.5,-3)--(-3,-2);
\draw [dashed,red] (-3,-3)--(-2,-1);
\draw [dashed,red] (-3,-3)--(-5,-1);
\draw [dashed,red] (-3,-3)--(3,1);
\draw [dashed,red] (-3,-2)--(-3,-0.5);
\draw [dashed,red](3,-2)--(3,1);
\draw [dashed,red] (-2,-0.5)--(-2,-1);
\draw [dashed,red] (-2,-0.5)--(2,0);
\draw [dashed,red] (2,-0.5)--(2,-1);
\draw [dashed,red] (-3.5,0.75)--(-3,1);
\draw [dashed,red] (-3.5,0.75)--(2,-1);
\draw [dashed,red] (4,-0.5)--(2,0);
\draw [dashed,red] (4,-0.5)--(3,-2);

\draw (-3,1)--(3,1);
\draw (-2,0)--(-2,1);
\draw (2,0)--(2,1);
\draw [fill=black](-3,1) circle (0.05cm)node [above]{$\Theta_{\infty,1}$};
\draw [fill=black](-2,1) circle (0.05cm)node [above]{$\Theta_{\infty,2}$};
\draw [fill=black](0,1) circle (0.05cm)node [above]{$\Theta_{\infty,3}$};
\draw [fill=black](2,1) circle (0.05cm)node [above]{$\Theta_{\infty,4}$};
\draw [fill=black](3,1) circle (0.05cm)node [above]{$2P_3$};
\draw (3.5,1) circle(0.00001cm) node [right] {$w=0$};
\draw [fill=black](-2,0) circle (0.05cm)node [right]{$P_3$};
\draw [fill=black](2,0) circle (0.05cm)node [right]{$\Theta_{\infty,5}$};

\draw [fill=blue](-4,-2) circle (0.05cm)node [below]{$\Theta_{3,1}$};
\draw [fill=blue](-5,-2) circle (0.05cm) node [below]{$A_2$};
\draw [fill=blue](-6,-2) circle (0.05cm)node [below]{$\Theta_{-3,1}$};
\draw [fill=blue](-5,-1) circle (0.05cm)node [right]{$\Theta_{-1,3}$};
\draw [fill=blue](-5,-3) circle (0.05cm)node [right]{$\Theta_{1,3}$};
\draw (-5.5,-3.5) circle (0.0001cm) node [below] {$w=-1/4$}; 

\draw [dashed,red] (-4,-2)--(-2,-0.5);
\draw [dashed,red] (-5,-1)--(-3,-0.5);
\draw (-4,-2)--(-6,-2);
\draw (-5,-3)--(-5,-1);
\end{tikzpicture}%
\caption{Fibration $36
$}
\label{fig:Grw1}
\end{center}
\end{figure}

Denote $Q_1=(0,0)$, $Q_2=(w(4w+1),0)$, $Q_3=(-w^2(4w+1),0)$ the two-torsion sections and $W_1=(-4w^3,-2w^3(2w+1))$ an infinite section of $E_u$.

On the Figures \ref{fig:Grw} and \ref{fig:Grw1}, in red bullets, can be viewed the following sections of the new fibration:
\[
\Theta_{1,2},\Theta_{1,4},\Theta_{-1,2},\Theta_{-1,4}
\]
and also 
\[
P_3+A_{23},2P_3+A_{23},P_3+A_{22},2P_3+A_{22}.
\]
The correspondence between sections of fibrations $\#50$ and $\#36$ can be settled by the transformation $\phi$. Recall that the components $\Theta_{i,j}$ are obtained by blowing up. For example the section $P_3=(x=0,y=0)$ intersects the component $\Theta_{1,2}$, so this component defined by $x=(u-1)^2x_2,y=(u-1)^2y_2$ satisfies $y_2=0$. It follows that the point $W_1$ corresponds to  $\Theta_{1,2}$ and 
 the $0$ section of the new fibration  to $\Theta_{1,4}.$ For all results see Table \ref{table:height36}.

\subsection{Heights of sections}

The heights of sections of the new fibration are computed with the help of the graph. For example, we can see on Figure \ref{fig:Grw1} that the section $2P_{3}+A_{22}$ intersects $\Theta_{1,4}$ (the zero section), $\Theta_{\infty,1}$ ($I_2^*$ for $w=0$), $\Theta_{-1,1}$ ($I_2^*$ for $w=\infty$), $\Theta_{-3,1}$ ($I_0^*$ for $w=-\frac{1}{4}$) and $\Theta_{0,1}$ ($I_2$ for $w=-1$). The respective contributions are then computed with Table \ref{table:cont} and from formula (\ref{E:S2}) it follows 

\[
h(2P_3+2A_{22})=4+2-(3/2+3/2+1+1/2)=3/2.
\]

Since the height of this section is equal to $\frac{3}{2}$, according to formula (\ref{F:disc}), it generates the Mordell-Weil lattice.
The results are summarized on Table \ref{table:height36}.

\begin{table}
\begin{center}
\caption{Heights for sections of fibration \#36} \label{table:height36}
\begin{tabular}
[c]{|c|cccccccc|}%
\hline
Contr. on & ${\scriptstyle \Theta_{1,4}}$ & ${\scriptstyle \Theta_{1,2}}$ & ${\scriptstyle \Theta_{-1,2}}$ & ${\scriptstyle \Theta
_{-1,4}}$ & ${\scriptstyle P_{3}+A_{22}}$ & ${\scriptstyle 2P_{3}+A_{22}}$ & ${\scriptstyle P_{3}+A_{23}}$ & ${\scriptstyle 2P_{3}+A_{23}}%
$\\
\hline
$I_{2}^{\ast}\, {\scriptstyle w=0}$ & $0$ & $\frac{3}{2}$ & $\frac{3}{2}$ & $0$ & $1$ &
$\frac{3}{2}$ & $1$ & $\frac{3}{2}$\\
$I_{2}^{\ast}\,{\scriptstyle w=\infty}$ & $0$ & $1$ & $\frac{3}{2}$ & $\frac{3}{2}$ &
$\frac{3}{2}$ & $\frac{3}{2}$ & $0$ & $1$\\
$I_{0}^{\ast}\,{\scriptstyle w=\frac{-1}{4}}$ & $0$ & $0$ & $1$ & $1$ & $1$ & $1$ & $1$ &
$1$\\
$I_{2}\, {\scriptstyle w=-1}$ & $0$ & $0$ & $0$ & $0$ & $\frac{1}{2}$ & $\frac{1}{2}$ &
$\frac{1}{2}$ & $\frac{1}{2}$\\
$height$ & $0$ & $\frac{3}{2}$ & $0$ & $\frac{3}{2}$ & $0$ & $\frac{3}{2}$ &
$\frac{3}{2}$ & $0$\\
& ${\scriptstyle 0}$ & ${\scriptstyle W_{1}}$ & ${\scriptstyle Q_{1}}$ & ${\scriptstyle Q_{1}+W_{1}}$ & ${\scriptstyle Q_{2}}$ & ${\scriptstyle Q_{2}+W_{1}}$ &
${\scriptstyle Q_{3}+W_{1}}$ & ${\scriptstyle Q_{3}}$\\
\hline%
\end{tabular}
\end{center}
\end{table}    

\section{Fibration $\#40$}

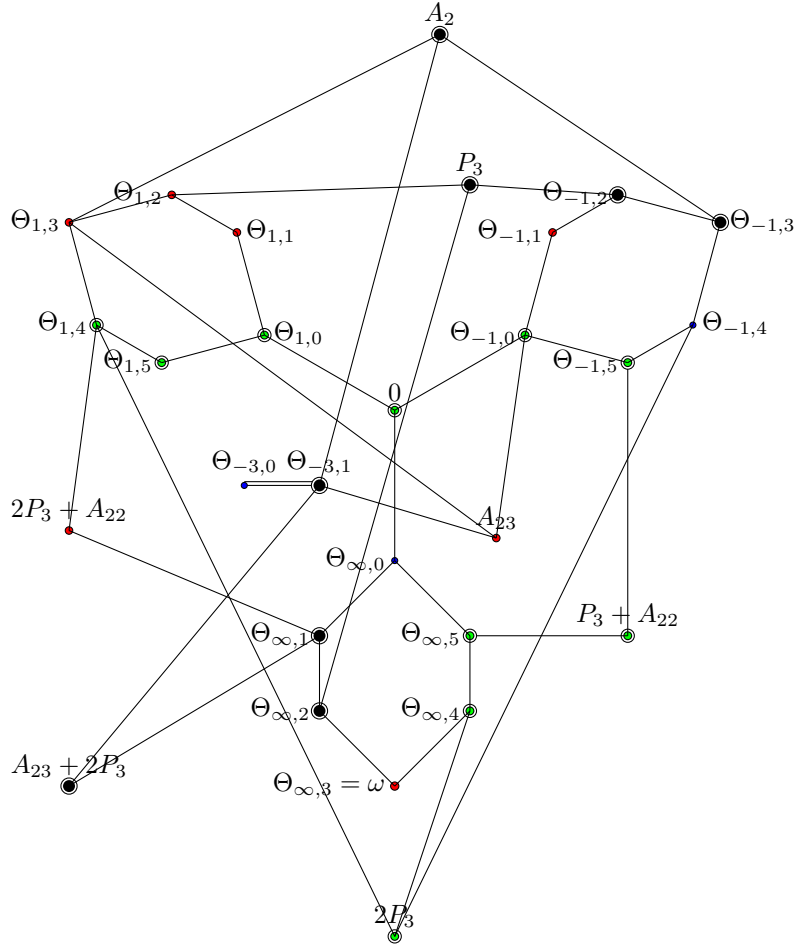
\begin{figure}
\begin{center}
\begin{tikzpicture}[scale=1]
\unitlength 1cm
\draw [double,fill=green](0,0) circle  (0.07cm)node [above] {$0$};

\draw [double,fill=black] (1,3) circle (0.09cm) node [above] {$P_3$};

\draw [double,fill=green](0,-7) circle (0.07cm) node [above] {$2P_3$};

\draw [double,fill=black](0.6,5) circle (0.09cm) node [above] {$A_2$};


\draw [double,fill=black] (-4.33,-5) circle (0.09cm) node [above] {$A_{23}+2P_3$};

\draw [double,fill=green] (3.098,-3) circle (0.07cm) node [above] {$P_3+A_{22}$};

\draw (3.098,-3)--(3.098,0.6339);
\draw (3.098,-3)--(1,-3);

\draw [fill=blue] (0,-2)circle (0.04cm) node[left]{$\Theta_{\infty,0}$};
\draw [fill=red] (0,-5)circle (0.055cm) node[left]{$\Theta_{\infty,3}=\omega$};
\draw [double,fill=green] (1,-3)circle (0.07cm) node[left]{$\Theta_{\infty,5}$};
\draw [double,fill=green] (1,-4)circle (0.07cm) node[left]{$\Theta_{\infty,4}$};
\draw [double,fill=black] (-1,-4)circle (0.09cm) node[left]{$\Theta_{\infty,2}$};
\draw [double,fill=black] (-1,-3)circle (0.09cm) node[left]{$\Theta_{\infty,1}$};
\draw (0,-2)--(1,-3);
\draw (1,-3)--(1,-4);
\draw (1,-4)--(0,-5);
\draw  (0,-5)--(-1,-4);
\draw (-1,-4)--(-1,-3);
\draw (-1,-3)--(0,-2);
\draw [double,fill=green] (1.732,1)circle (0.07cm) node[left]{$\Theta_{-1,0}$};
\draw [double,fill=black] (4.33,2.5)circle (0.09cm) node[right]{$\Theta_{-1,3}$};
\draw [fill=red] (2.098,2.366)circle (0.05cm) node[left]{$\Theta_{-1,1}$};
\draw [double,fill=black] (2.9641,2.866)circle (0.09cm) node[left]{$\Theta_{-1,2}$};
\draw [fill=blue] (3.964,1.1339)circle (0.04cm) node[right]{$\Theta_{-1,4}$};
\draw [double,fill=green] (3.098,0.6339)circle (0.07cm) node[left]{$\Theta_{-1,5}$};
\draw (1.732,1)--(2.098,2.366);
\draw (2.098,2.366)--(2.9641,2.866);
\draw (2.9641,2.866)--(4.33,2.5);
\draw (4.33,2.5)--(3.964,1.1339);
\draw (3.964,1.1339)--(3.098,0.6339);
\draw (3.098,0.6339)--(1.732,1);

\draw [double,fill=green] (-1.732,1)circle (0.07cm) node[right]{$\Theta_{1,0}$};
\draw [fill=red] (-4.33,2.5)circle (0.05cm) node[left]{$\Theta_{1,3}$};
\draw [fill=red] (-2.098,2.366)circle (0.05cm) node[right]{$\Theta_{1,1}$};
\draw [fill=red] (-2.9641,2.866)circle (0.05cm) node[left]{$\Theta_{1,2}$};
\draw [double,fill=green] (-3.964,1.1339)circle (0.07cm) node[left]{$\Theta_{1,4}$};
\draw [double,fill=green] (-3.098,0.6339)circle (0.07cm) node[left]{$\Theta_{1,5}$};

\draw (-1.732,1)--(-2.098,2.366);
\draw (-2.098,2.366)--(-2.9641,2.866);
\draw (-2.9641,2.866)--(-4.33,2.5);
\draw (-4.33,2.5)--(-3.964,1.1339);
\draw (-3.964,1.1339)--(-3.098,0.6339);
\draw (-3.098,0.6339)--(-1.732,1);

\draw (0,0)--(0,-2);
\draw (0,0)--(1.732,1);
\draw (0,0)--(-1.732,1);

\draw (0,-7)--(1,-4);
\draw (0,-7)--(-3.964,1.1339);
\draw (0,-7)--(3.964,1.1339);

\draw (1,3)--(-1,-4);
\draw (1,3)--(-2.9641,2.866);
\draw (1,3)--(2.9641,2.866);

\draw (0.6,5)--(4.33,2.5);
\draw (0.6,5)--(-4.33,2.5);



\draw (-2,-1)--(-1,-1);
\draw (-2,-0.95)--(-1,-0.95);
\draw [fill=blue] (-2,-1)circle (0.04cm) node[above]{$\Theta_{-3,0}$};
\draw [double,fill=black] (-1,-1)circle (0.09cm) node[above]{$\Theta_{-3,1}$};


\draw(-1,-1)--(-4.33,-5);
\draw(-4.33,-5)--(-1,-3);


\draw (0.6,5)--(-1,-1);

\draw [fill=red](1.35,-1.7) circle (0.05cm) node [above]{$A_{23}$};
\draw (1.35,-1.7)--(1.732,1);
\draw (1.35,-1.7)--(-1,-1);
\draw (1.35,-1.7)--(-4.33,2.5);

\draw [fill=red](-4.33,-1.6) circle (0.05cm) node [above]{$2P_3+A_{22}$};
\draw (-4.33,-1.6)--(-3.964,1.1339);
\draw (-4.33,-1.6)--(-1,-3);

\end{tikzpicture}%
\caption{Fibration $40$}
\label{fig:Grp}
\end{center}
\end{figure}

The two divisors
\begin{align*}
D  & =A_{2}+\Theta_{-1,3}+\Theta_{-1,2}+P_{3}+\Theta_{\infty,2}+\Theta
_{\infty,1}+\left(  A_{23}+2P_{3}\right)  +\Theta_{-3,1}\\
D^{\prime}  & =\Theta_{1,4}+\Theta_{1,5}+\Theta_{1,0}+0+\Theta_{-1,0}%
+\Theta_{-1,5}+\left(  P_{3}+A_{22}\right)  +\Theta_{\infty,5}+\Theta
_{\infty,4}+2P_{3}%
\end{align*}
can be viewed as two singular fibers of an elliptic fibration with
elliptic parameter $p$ determined as explained in 9.1.

First we search on $E_{u}$ a function $g$ with three simple poles at
$0,(2P_{3})$ and $P_{3}+A_{22}$ \ and three zeros at $P_{3},A_{2}$ and
$A_{23}+2P_{3}.$ Taking
\[
g=r+\frac{y-y_{P_{3}}}{x-x_{P_{3}}}+s\frac{y-y_{2P_{3}+A_{22}}}{x-x_{2P_{3}%
+A_{22}}}%
\]
and choosing $r$ and $s$ satisfying $g\left(  A_{2}\right)  =g\left(
A_{23}+2P_{3}\right)  =0$, we get $r=-u+1,s=-\frac{u-1}{u+1}.$ Finally to insure poles on $D^{\prime}$ set $p=\frac{g}{u-1}$ so
\[
p=\frac{\left(  2x+\left(  u^{2}-1\right)  ^{2}\right)  y-\left(
u^{2}-1\right)  x^{2}}{\left(  u^{2}-1\right)  x\left(  x+\left(  u+1\right)
\left(  u-1\right)  ^{2}\right)  }.
\]

We can remark that $p$ can also be obtained from the fibration \#36 and the
parameter

\[
p=\frac{-Y}{w\left(  X-w\left(  1+4w\right)  \right)  }.%
\]

The usual transformations leading to a Weierstrass equation are summarized in the birational map $\phi:(\left(  \mathfrak{x,y,}p\right)\mapsto \left(  x,y,u\right) $
with%

\begin{align*}
x  &  =\frac{-G_{3}^{2}G_{4}^{2}}{p^{2}\mathfrak{x}\left(  \mathfrak{x}^{2}-p\left(
p^{2}+4p-1\right)  \mathfrak{x}-4p^{3}\right)  ^{4}}\quad y=\frac{-G_{2}G_{4}%
^{2}G_{3}^{3}}{p^{2}\mathfrak{x}\left(  \mathfrak{x}^{2}-p\left(  p^{2}+4p-1\right)
\mathfrak{x}-4p^{3}\right)  ^{6}}\\
u  &  =\frac{-G_{1}}{p\left(  \mathfrak{x}^{2}-p\left(  p^{2}+4p-1\right)
\mathfrak{x}-4p^{3}\right)  }%
\end{align*}

where%

\begin{align*}
G_{1}  =\left(  \mathfrak{x}-2p^{2}\right)  \mathfrak{y}+p^{2}\left(
p^{2}+1\right)  \mathfrak{x}-4p^{4}, &  G_{2}=\left(  p+1\right)  \mathfrak{y}%
+\mathfrak{x}^{2}-p^2\left(  p+3\right)  \mathfrak{x},\\
G_{3}  =\left(  2p^{2}-\mathfrak{x}\right)  \mathfrak{y}-p\mathfrak{x}^{2}%
+2p^{2}\left(  2p-1\right)  \mathfrak{x}+8p^{4}, & G_{4}=\left(  2p^{2}%
-\mathfrak{x}\right)  \mathfrak{y}+p\mathfrak{x}^{2}-2p^{3}\left(  p+2\right)  \mathfrak{x}.%
\end{align*}

We find then the following Weierstrass equation
\begin{equation}
\mathfrak{y}^{2}-\left(  p^{2}+1\right)  \mathfrak{yx}+4p^{2}\mathfrak{y}=\mathfrak{x}\left(
\mathfrak{x}-p^{2}\right)  \left(  \mathfrak{x}-4p^{2}\right).  \label{Eq:Ep}%
\end{equation}

We denote $V_{1}=\left(  2p\left(  p-1\right)  ,2p\left(  p-1\right)  \right)
$ and $V_{2}=\left(  0,-4p^{2}\right).$

The first and last line of Table \ref{table:height40} are computed using $\phi$ and also $\omega=\Theta_{\infty,3}$  the zero of the new fibration. 

From the graph (Figure \ref{fig:Grp}) we obtain the index of the component of the singular fibers ($I_8$ and $I_{10}$) which a given section $S$ meets (line 2 and 3 of  Table \ref{table:height40}). 
Then we compute the heights as explained in 9.3. From formula (\ref{E:S1}), it follows 
$<\Theta_{1,1},2P_{3}+A_{22}>=\frac{1}{20}$. Thus the height matrix of $\Theta_{1,1}$ and $2P_{3}+A_{22}$ has determinant $3/20.$; we recover the result:

{\it{ The two sections $V_1$ and $V_2$ generate the Mordell-Weil lattice}}.





\begin{table}[b]
\begin{center}
\caption{Heights for sections of fibration \#40} \label{table:height40}
\begin{tabular}
[c]{|c|cccccccccc|}%
\hline
\text{sect.}  & $ {\scriptstyle \Theta_{1,1}} $ & ${\scriptstyle \Theta_{1,2}}$ &
${\scriptstyle \Theta_{1,3}}$ & ${\scriptstyle \Theta_{-1,1}}$ & ${\scriptstyle  \Theta_{3,1}}$ & $ {\scriptstyle \Theta_{0,1}}$ & ${\scriptstyle A_{22}}$ & ${\scriptstyle A_{23}}$ & ${\scriptstyle 2P_{3}+A_{22}}$ & ${\scriptstyle 2P_{3}+A_{2}}$\\
\hline
$I_{8}$ & 2 & 7 & 4 & 6 & 4 & 2& 5 & 3 & 1 & 3\\
$I_{10}$ & 4 & 8 & 2 & 6 & 8 & 8& 4 & 6 & 2 & 0\\
${\scriptstyle ht}$ & $\frac{1}{10}$ & $\frac{61}{40}$ & $\frac{4}{10}$ & $\frac{1}{10}$ & $\frac%
{4}{10}$ &$ \frac{9}{10}$& $\frac{69}{40}$ & $\frac{69}{40}$ & $\frac{61}{40}$ & $\frac{17}{8}$\\
& ${\scriptstyle V_{2}}$ & ${\scriptstyle V_{1}-V_{2}}$ & ${\scriptstyle -2V_{2}}$ & ${\scriptstyle -V_{2}}$ & ${\scriptstyle 2V_{2}}$ & ${\scriptstyle -3V_{2}}$& ${\scriptstyle V_{1}-2V_{2}}$ & ${\scriptstyle V_{1}+V_{2}}$ & ${\scriptstyle V_{1}}$ & ${\scriptstyle V_{1}-3V_{2}}$\\%
\hline
\end{tabular}
\end{center}
\end{table}

\section{Fibration $\#40$ bis}
\begin{figure}
\begin{center}
\begin{tikzpicture}[scale=1]
\unitlength 1cm
\draw [fill=green](0,0) circle  (0.07cm)node [above] {$0$};
\draw (0,0) circle (0.15cm);
\draw [fill=red] (0,3) circle (0.05cm) node [above] {$P_3$};

\draw [fill=red](0,-7) circle (0.05cm) node [above] {$2P_3$};

\draw [fill=green](0.6,5) circle (0.07cm) node [above] {$A_2$};
\draw (0.6,5) circle(0.15cm);

\draw [double,fill=black] (-4.33,-5) circle (0.07cm) node [above] {$A_{23}+2P_3$};

\draw [double,fill=black] (3.098,-3) circle (0.07cm) node [above] {$P_3+A_{22}$};

\draw (3.098,-3)--(3.098,0.6339);
\draw (3.098,-3)--(1,-3);

\draw [fill=blue] (0,-2)circle (0.05cm) node[left]{$\Theta_{\infty,0}$};
\draw [double,fill=black] (0,-5)circle (0.07cm) node[left]{$\Theta_{\infty,3}$};
\draw [double,fill=black] (1,-3)circle (0.07cm) node[left]{$\Theta_{\infty,5}$};
\draw [double,fill=black] (1,-4)circle (0.07cm) node[left]{$\Theta_{\infty,4}$};
\draw [double,fill=black] (-1,-4)circle (0.07cm) node[left]{$\Theta_{\infty,2}$};
\draw [double,fill=black] (-1,-3)circle (0.07cm) node[left]{$\Theta_{\infty,1}$};
\draw (0,-2)--(1,-3);
\draw (1,-3)--(1,-4);
\draw (1,-4)--(0,-5);
\draw  (0,-5)--(-1,-4);
\draw (-1,-4)--(-1,-3);
\draw (-1,-3)--(0,-2);
\draw [fill=green] (1.732,1)circle (0.07cm) node[left]{$\Theta_{-1,0}$};
\draw (1.732,1) circle(0.15cm);
\draw [fill=green] (4.33,2.5)circle (0.07cm) node[right]{$\Theta_{-1,3}$};
\draw (4.33,2.5) circle(0.15cm);
\draw [fill=green] (2.098,2.366)circle (0.07cm) node[left]{$\Theta_{-1,1}$};
\draw (2.098,2.366) circle(0.15cm);
\draw [fill=green] (2.9641,2.866)circle (0.07cm) node[left]{$\Theta_{-1,2}$};
\draw (2.9641,2.866) circle(0.15cm);
\draw [fill=red] (3.964,1.1339)circle (0.05cm) node[right]{$\Theta_{-1,4}$};
\draw [fill=red] (3.098,0.6339)circle (0.05cm) node[left]{$\Theta_{-1,5}$};
\draw (1.732,1)--(2.098,2.366);
\draw (2.098,2.366)--(2.9641,2.866);
\draw (2.9641,2.866)--(4.33,2.5);
\draw (4.33,2.5)--(3.964,1.1339);
\draw (3.964,1.1339)--(3.098,0.6339);
\draw (3.098,0.6339)--(1.732,1);

\draw [fill=green] (-1.732,1)circle (0.07cm) node[right]{$\Theta_{1,0}$};
\draw (-1.732,1) circle(0.15cm);
\draw [fill=green] (-4.33,2.5)circle (0.07cm) node[left]{$\Theta_{1,3}$};
\draw (-4.33,2.5) circle(0.15cm);
\draw [fill=red] (-2.098,2.366)circle (0.05cm) node[right]{$\Theta_{1,1}=\omega$};
\draw [fill=red] (-2.9641,2.866)circle (0.05cm) node[left]{$\Theta_{1,2}$};
\draw [fill=green] (-3.964,1.1339)circle (0.07cm) node[left]{$\Theta_{1,4}$};
\draw (-3.964,1.1339) circle(0.15cm);
\draw [fill=green] (-3.098,0.6339)circle (0.07cm) node[left]{$\Theta_{1,5}$};
\draw (-3.098,0.6339) circle(0.15cm);

\draw (-2.098,2.366)--(-4.33,-5);
\draw (-1.732,1)--(-2.098,2.366);
\draw (-2.098,2.366)--(-2.9641,2.866);
\draw (-2.9641,2.866)--(-4.33,2.5);
\draw (-4.33,2.5)--(-3.964,1.1339);
\draw (-3.964,1.1339)--(-3.098,0.6339);
\draw (-3.098,0.6339)--(-1.732,1);

\draw (0,0)--(0,-2);
\draw (0,0)--(1.732,1);
\draw (0,0)--(-1.732,1);

\draw (0,-7)--(1,-4);
\draw (0,-7)--(-3.964,1.1339);
\draw (0,-7)--(3.964,1.1339);

\draw (0,3)--(-1,-4);
\draw (0,3)--(-2.9641,2.866);
\draw (0,3)--(2.9641,2.866);

\draw (0.6,5)--(4.33,2.5);
\draw (0.6,5)--(-4.33,2.5);
\draw (0.6,5)--(0,-2);


\draw (1.35,-1.7)--(2.35,-1.7);
\draw (1.35,-1.65)--(2.35,-1.65);
\draw [fill=blue] (1.35,-1.7)circle (0.07cm) node[above]{$\Theta_{0,0}$};
\draw [double,fill=black] (2.35,-1.7)circle (0.07cm) node[above]{$\Theta_{0,1}$};

\draw [fill=red] (-1.8,-1)circle (0.05cm) node [above]{$P_3+A_2$};
\draw (-1.8,-1)--(-3.098,0.6339);
\draw (-1.8,-1)--(3.098,0.6339);
\draw (-1.8,-1)--(-1,-4);

\draw (1.35,-0.7)--(2.35,-0.7);
\draw (1.35,-0.65)--(2.35,-0.65);
\draw [fill=red] (1.35,-0.7)circle (0.05cm) node[above]{$\Theta_{3,0}$};
\draw [fill=red] (2.35,-0.7)circle (0.05cm) node[above]{$\Theta_{3,1}$};

\draw(-4.33,-5)--(-1,-3);

\draw (2.35,-0.7)--(3.098,-3);
\draw (2.35,-1.7)--(3.098,-3);

\draw (-4.33,-5)--(2.35,-1.7);

\end{tikzpicture}%
\caption{Fibration $40$ bis}
\label{fig:Grt}
\end{center}
\end{figure}
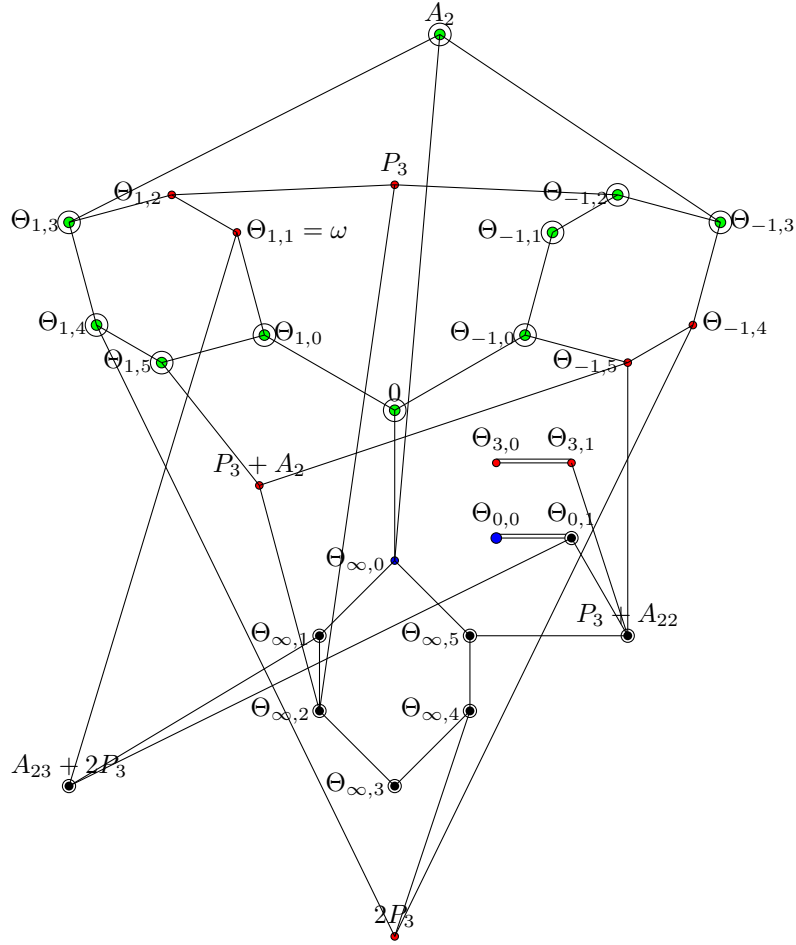

The two divisors 
\begin{align*}
D=\Theta_{\infty,1}+\Theta_{\infty,2}+\Theta_{\infty,3}+\Theta_{\infty,4}+\Theta_{\infty,5}+(P_3+A_{22})+\Theta_{0,1}
+(A_{23}+2P_3)\\
D'=\Theta_{-1,2}+\Theta_{-1,3}+A_2+\Theta_{1,3}+\Theta_{1,4}+\Theta_{1,5}+\Theta_{1,0}+0+\Theta_{-1,0}+\Theta_{-1,1}
\end{align*}
define two singular fibers of an elliptic fibration with elliptic parameter
\[
t=\frac{1}{2}\frac{2y+\left(  u-3\right)  \left(  u-1\right)  x-\left(
u-1\right)  ^{3}\left(  u+1\right)  ^{2}}{\left(  u^{2}-1\right)  \left(
x+\frac{1}{4}\left(  u^{2}-1\right)  ^{2}\right)  },
\]
and Weierstrass equation
\[
\mathfrak{y}^{2}+2\left(  t^{2}-1\right)  \mathfrak{yx}-2t^{2}\mathfrak
{y}=\mathfrak{x}\left(  \mathfrak{x}+t^{2}\right)  \left(  \mathfrak{x}%
+4t^{2}\right)
\]

with the following birational transformations%

\[
\left(  \mathfrak{x},\mathfrak{y},t\right) \mapsto \left(  x,y,u\right)  
\]%
\[
x=-8\frac{\mathfrak{y}(\mathfrak{x}+1)^{2}(\mathfrak{x}+4t^{2})H_{1}^{2}%
}{((2t+1)\mathfrak{y}+(\mathfrak{x}+4t^{2})(\mathfrak{x}-t))^{4}}%
\]%
\[
y=16\frac{(\mathfrak{x}+1)^{2}(\mathfrak{x}+4t^{2})(2\mathfrak{y}%
+4t^{2}\mathfrak{x}+\mathfrak{x}^{2})H_1^{4}}{((2t+1)\mathfrak{y}+(\mathfrak
{x}+4t^{2})(\mathfrak{x}-t))^{6}}%
\]%
\[
u=-\frac{(2t+1)\mathfrak{y}-(\mathfrak{x}+4t^{2})(\mathfrak{x}+2+t)}%
{(2t+1)\mathfrak{y}+(\mathfrak{x}+4t^{2})(\mathfrak{x}-t)}%
\]
where
\[
H_{1}=-(2t+1)\mathfrak{y}+(t+1)(\mathfrak{x}+4t^{2}).
\]

Notice also the relations
\[u-1=\frac{2H_1}{(2t+1)\mathfrak{y}+(\mathfrak{x}+4t^2)(\mathfrak{x}-t)}\]
and
\[u+1=\frac{2(\mathfrak{x}+1)(\mathfrak{x}+4t^2)}{(2t+1)\mathfrak{y}+(\mathfrak{x}+4t^2)(\mathfrak{x}-t)}.\]

Let $Z_{1}=\left(  0,0\right)  $ and $Z_{5}=\left(  -1,\left(  2t-1\right)
\left(  t+1\right)  \right)  $.

It follows from the previous formulae that the $0$ section of the new fibration
corresponds to $u=1$ and looking at $x/(u-1)$ and at the graph we find that the
$0$ section corresponds to $\Theta_{1,1}.$ The correspondence between the sections of the two fibrations can be also derived and is shown on Table \ref{table:height40bis}. On the same table are quoted the contributions and the heights of sections computed with the graph.
 
Moreover we find $<\Theta_{-1,4},P_{3}>=2+\Theta_{1,1}\cdot P_{3}+\Theta
_{1,1}\cdot\Theta_{-1,4}-\Theta_{-1,4}\cdot P_{3}-\frac{5\times4}{10}=0.$

Thus the height matrix of the two sections $\Theta_{-1,4}$ and $P_{3}$ is diagonal
with determinant $\frac{3}{20},$ so $Z_1$ and $Z_5$ generate the Mordell-Weil lattice.

\begin{table}
\begin{center}
\caption{Heights for sections of fibration \#40 bis} \label{table:height40bis}
\begin{tabular}
[c]{|c|cccccccccc|}%
\hline
\text{sect} & ${\scriptstyle \Theta_{1,1}}$ &$ {\scriptstyle \Theta_{1,2}}$ & ${\scriptstyle \Theta_{-1,4}}$ & ${\scriptstyle \Theta_{-1,5}}$ &
${\scriptstyle \Theta_{3,0}}$ & ${\scriptstyle \Theta_{3,1}}$ & ${\scriptstyle P_{3}}$ &${\scriptstyle  2P_{3} }$& ${\scriptstyle P_{3}+A_{2}}$ & ${\scriptstyle 2P_{3}+A_{2}}$\\
\hline
$I_{8}$ & $0 $& $2$ &$ 0 $& $2$ & $0$ & $2 $& $6$ &$ 4 $& $6$ & $4$\\
$I_{10}$ & $0 $&$ 7 $&$ 5$ & $2 $& $1$ & $6$ &$ 4 $&$ 8 $&$ 9 $&$ 3$\\
$ht $&$ 0 $& $\frac{12}{5} $& $\frac{3}{2}$ &$ \frac{9}{10}$ & $\frac{31}{10}$ & $\frac
{1}{10}$ & $\frac{1}{10}$ &$ \frac{4}{10} $&$ \frac{8}{5} $& $\frac{19}{10}$\\
&$ {\scriptstyle 0}$ & ${\scriptstyle 3Z_{1}-Z_{5}}$ & ${\scriptstyle Z_{5}} $&$ {\scriptstyle 3Z_{1}}$ & ${\scriptstyle 4Z_{1}-Z_{5}}$ &${\scriptstyle  -Z_{1}} $&${\scriptstyle  Z_{1}} $&$ {\scriptstyle 2Z_{1}}$ &$
{\scriptstyle Z_{1}-Z_{5}}$ &${\scriptstyle  2Z_{1}-Z_{5}}$\\
\hline
\end{tabular}
\end{center}
\end{table}


\begin{thebibliography}{10}

\bibitem[BGHLMSW]{BGL} M.J. Bertin, A. Garbagnati, R. Hortsch, O. Lecacheux, M. Mase, C. Salgado, U. Whitcher {\it Classifications of Elliptic Fibrations of a Singular K3 Surface}, in Women in Numbers Europe, 17--49, Research Directions in Number Theory, Association for Women in Mathematics Series, Springer, 2015. 

\bibitem[BL]{BL} M.J. Bertin, O. Lecacheux {\it Elliptic fibrations on the modular surface associated to $\Gamma_1(8)$},  in Arithmetic and geometry of K3 surfaces and Calabi-Yau threefolds, 153--199, Fields Inst. Commun., {\bf 67}, Springer, New York, 2013.

\bibitem[Cr]{Cr} J. Cremona, {\it Computing in component groups of elliptic curves},
ANTS VIII Proceedings: A. van der Poorten, A. Stein (eds.), ANTS VIII, Lecture Notes in Computer Science 5011 (2008), 118--124.

\bibitem[CS]{CS} J.H. Conway, N.J.A. Sloane, Sphere Packings, Lattices and Groups, Springer-Verlag (1993).

\bibitem[El]{El} N. Elkies, Private communication.

\bibitem[IS1]{IS1} M. Ishii, D. Sagaki, H. Shimakura, {\it Automorphisms of Niemeier lattices
for Miyamoto $\mathbb Z 3$-orbifold construction}, Mathematische Zeitschrift
June 2015, Volume 280, Issue 1, 55--83.


\bibitem[IS2]{IS2} D. Sagaki, H. Shimakura {\it Application of a  Z 3-orbifold construction to the lattice vertex operator algebras associated to Niemeier lattices}, Trans. Amer. Math. Soc. 368 (2016), 1621--1646.

\bibitem[Nis]{Nis} K. Nishiyama, {\it The Jacobian fibrations on some K3 surfaces and their Mordell-Weil groups}, Japan. J. Math. (N.S.) {\bf 22 }(1996) 293--347.

\bibitem[ScSh]{ScSh} M. Sch\" {u}tt,  T. Shioda,  {\it Elliptic surfaces}, Algebraic geometry in East Asia--Seoul 2008,  Adv. Stud. Pure Math., {\bf 60}, Math. Soc. Japan, Tokyo, (2010) 51--160.

\bibitem[Si]{Si} J. H. Silverman, {\it Computing heights on elliptic curves}, Math. Comp. {\bf 51} (183) (1988) 339--358.



\end{thebibliography}
\end{document}